\documentclass[smallcondensed, numbook]{svjour3}
\usepackage{amsmath,amssymb,mathrsfs}
\usepackage{mathtools}
\usepackage{bm}
\usepackage{multirow}   
\usepackage{booktabs}    
\usepackage{makecell}
\usepackage{verbatim}
\usepackage{float}
\usepackage{siunitx}
\usepackage{graphicx}
\usepackage[caption=false]{subfig}
\usepackage{url}

\newtheorem{thm}{Theorem}[section]

\renewcommand{\tablename}{Table}
\renewcommand{\thetable}{\arabic{table}}
\renewcommand{\figurename}{Figure}
\renewcommand{\thefigure}{\arabic{figure}}
\numberwithin{equation}{section}
\usepackage{etoolbox}
\makeatletter
\renewcommand{\thetable}{\thesection.\arabic{table}}
\renewcommand{\thefigure}{\thesection.\arabic{figure}}
\pretocmd{\section}{\setcounter{table}{0}\setcounter{figure}{0}}{}{}
\makeatother

\begin{document}

\title{Rigorous analysis of the time-splitting methods for the semiclassical Dirac equation}

\titlerunning{Rigorous Analysis of the time-splitting methods for the semiclassical Dirac equation}

\author{He Wang \and Jia Yin}

\institute{%
\at School of Mathematical Sciences, Fudan University, Shanghai 200433, China \\
\email{hwang24@m.fudan.edu.cn (He Wang),\, jiayin@fudan.edu.cn (Jia Yin)}
}

\date{}

\maketitle

\begin{abstract}
We provide rigorous error analysis of the mass-preserving time-splitting methods for solving the semiclassical Dirac equation. The scaled Planck constant $\epsilon$ in the equation gives rise to rapid oscillations in both space and time when $0<\epsilon\ll 1$ with wavelengths of order $O(\epsilon)$.
Rigorous error estimates reveal the precise dependence of the approximation errors on the time step $\tau$, the spatial mesh size $h$, and the parameter $\epsilon$. Specifically, the temporal error scales as $O\left(\tau/\epsilon^2\right)$ for the first-order splitting $S_1$ and as $O\left(\tau^2/\epsilon^3\right)$ for the second-order splitting $S_2$, while the spatial error scales as $O(h^m/\epsilon^m)$ for both methods, where $m$ is related to the regularity of the solution. In addition, we obtain error bounds for key physical observables, including the total probability density $\rho$ and the current density $\mathbf{J}$. Compared with finite difference time domain
(FDTD) methods, time-splitting approaches exhibit spectral accuracy in space and retain a relatively low computational cost. Furthermore, we demonstrate that higher accuracy can be achieved by employing the fourth-order compact time-splitting ($S_\text{4c}$) method. Numerical experiments are conducted to verify the reliability of the error estimates.

\end{abstract}

\section{Introduction}

\noindent
The Dirac equation, formulated by Paul Dirac in 1928, describes the relativistic quantum dynamics of spin-1/2 particles. The equation not only accounts for the fine structure of the hydrogen spectrum, but also successfully predicted the existence of antimatter, thus having a profound impact on particle physics \cite{dirac1928quantum}. In condensed matter systems, the massless Dirac equation governs the low‑energy electronic behavior of graphene, explaining its unusual linear dispersion and the associated quantum Hall effect  \cite{castro2009electronic3,novoselov2004electric2}. Furthermore, the Dirac equation provides the fundamental framework for understanding the topological surface states and associated dissipationless edge currents in topological insulators and topological semimetals \cite{Hasan2010Topological4,Young2012Dirac5}. In the field of quantum simulation, it has been employed to model physical phenomena in curved spacetime \cite{boada2011dirac7}. In strong‑field physics, the time‑dependent Dirac equation is essential for simulating highly relativistic electron dynamics under ultra‑intense laser fields (e.g., attosecond lasers), including highly non-linear quantum electrodynamical processes such as electron-positron pair production \cite{DiPiazza2012Strong6}. Among its various regimes, the Dirac equation in the semiclassical limit is of particular interest due to its distinctive physical and mathematical features, as it describes the transition from quantum to classical dynamics.

\noindent
According to the framework established in \cite{bao2017numerical,bao2019fourth}, the $d$-dimensional ($d = 1,2,3$) semiclassical Dirac equation admits the following formulation:

\begin{equation}\label{def-Dirac-complete}
    i \epsilon \partial_t \textbf{$\textbf{$\Psi$}$} = \left[ \left( -i \epsilon \sum_{j=1}^d \alpha_j \partial_j + \beta \right) +  \left( V(t, \textbf{x}) I_4 - \sum_{j=1}^dA_j(t, \textbf{x}) \alpha_j
    \right) \right] \textbf{$\Psi$},
\end{equation}
where $i = \sqrt{-1}$ is the imaginary unit, $\textbf{x} = (x_1, \cdots,x_d)$ represents the spatial coordinate, $t$ is time, $I_4$ denotes the $4\times 4$ identity matrix and the 4 $\times$ 4 matrices $\alpha_j$ ($j = 1$, $\cdots$, $d$), $\beta$ are defined as follows:

\begin{equation}
    \alpha_j = \begin{bmatrix}
        \textbf{0} & \sigma_j \\
        \sigma_j & \textbf{0}
    \end{bmatrix}, \qquad j = 1, \cdots, d,\qquad \beta = \begin{bmatrix}
        I_2 & \textbf{0} \\
        \textbf{0} & - I_2
    \end{bmatrix}.
\end{equation}
In the above expression, $I_2$ denotes the 2 $\times$ 2 identity matrix and $\sigma_j$  ($j = 1$, $\cdots$, $d$) represent the Pauli matrices which are given as

\begin{equation}
        \sigma_1 = \begin{bmatrix}
        0 & 1 \\
        1 & 0
    \end{bmatrix}, \qquad \sigma_2 = \begin{bmatrix}
        0 & -i \\
        i & 0
    \end{bmatrix}, \qquad \sigma_3 = \begin{bmatrix}
        1 & 0 \\
        0 & -1
    \end{bmatrix}.
\end{equation}

\noindent
Moreover, in the equation \eqref{def-Dirac-complete}, $\textbf{$\Psi$}\in\mathbb{C}^4$ represents the four-component spinor wave function, and $V, A_j \in \mathbb{R}$ ($j = 1$, $\cdots$, $d$) represent the electric and magnetic potentials, respectively. The parameter $\epsilon \in (0,1)$ is the scaled Planck constant, which influences the scale of oscillations in the solution. 

\noindent
In the cases $d=1,2$, as given in \cite{bao2016uniformly}, we can reduce the Dirac equation \eqref{def-Dirac-complete}  to the equation of the two-component wave function $\Phi$ :

\begin{equation}\label{def-Dirac-simple}   
    i \epsilon \partial_t \textbf{$\Phi$} = \left[ \left( -i \epsilon \sum_{j=1}^d \sigma_j \partial_j + \sigma_3 \right) +  \left( V(t, \textbf{x}) I_2 - \sum_{j=1}^dA_j(t, \textbf{x}) \sigma_j
    \right) \right] \textbf{$\Phi$},
\end{equation}

\noindent
where $\Phi = \left( \Psi_1, \Psi_4 \right)^T$ or $\left( \Psi_2, \Psi_3 \right)^T \in \mathbb{C}^2$. For simplicity, we only deal with the one- or two-dimensional (1D or 2D) semiclassical Dirac equation with the reduced form~\eqref{def-Dirac-simple} in the following discussion.

\noindent
In the classical regime where $\epsilon = 1$, there have been a variety of efficient numerical methods for the Dirac equation, such as the finite difference time domain (FDTD) methods \cite{classical-solution-fdtd1,classical-solution-fdtd2} and exponential wave integrator Fourier pseudospectral (EWI-FP) method \cite{bao2017numerical}. However, in the semiclassical regime, these methods may suffer from less efficiency due to the high frequency oscillation of the solution. Specificially, their accuracy typically depends on the discretization parameters resolving the small wavelength associated with $\epsilon$. In the semiclassical regime, i.e. $0 < \epsilon \ll 1$, the solution oscillates with the wavelength $O( \epsilon)$  in time and space. This high-frequency oscillation requires fine mesh size $h = O(\epsilon)$ and $\tau = O(\epsilon)$ , which makes direct numerical resolution prohibitively expensive. This motivates the study of numerical methods with improved $\epsilon$-scalability. Several efficient numerical methods for the semiclassical Dirac equation including LFFD, SIFD1 and SIFD2 \cite{fdtd_dirac}. Most FDTD schemes are conditionally stable, requiring a coupling between the spatial and temporal mesh sizes. In contrast, the CNFD method is unconditionally stable but entailing a significant computational cost as it requires solving a linear system at each time step.
To address the above difficulties, we apply the time-splitting methods to the semiclassical Dirac equation, which is used to solve the Schr\"{o}dinger equation in \cite{sp_SchrodingerEquation,sp-semiregime-SchrodingerEquation,sp-semilag-SchrodingerEquation}, the Stokes equation \cite{sp-Stokes-equation} and the Dirac equation in different regimes \cite{sp-dirac-uniform-error-bounds,bao2019fourth,yin_sp_time-dependent}. Specifically, we consider the first-order ($S_1$), the second-order ($S_2$) and the compact fourth-order ($S_{4c}$) splitting methods. Among these methods, $S_1$ and $S_2$ are mass-preserving, while $S_\text{4c}$ also conserves mass in 1D. Compared to FDTD schemes, the time-splitting methods are highly efficient and exhibit higher-order accuracy.

\noindent
The rest of this paper is structured as follows: In Section \ref{section2}, we introduce the time-splitting algorithms, including $S_1$, $S_2$ and $S_\text{4c}$, for solving the semiclassical Dirac equation and show the mass conservation property. In Section \ref{section3}, we give the rigorous error estimates of the time-splitting methods under some regularity conditions and prove the result for $S_1$. We conduct numerical experiments in Section \ref{section4} to verify the error estimates established. Finally, we conclude in Section \ref{section5}.

\begin{remark}
Throughout the paper, we employ the standard Sobolev spaces and their associated norms. Moreover, the notation
\[
A \lesssim B,
\]
indicates that there exists a generic constant $C>0$, independent of $\epsilon$, $\tau$, and $h$, such that
\[
|A| \leq C B.
\]

\end{remark}

\section{Time-splitting Spectral Methods} \label{section2}

\noindent
In this section, we apply the time-splitting methods to the semiclassical Dirac equation and prove the property of mass conservation. To employ spectral methods with Fourier basis which enables the efficient use of the Fast Fourier Transform (FFT), we impose periodic boundary conditions on the Dirac equation \eqref{def_Dirac_simplified}. This assumption is reasonable since the wave function usually vanishes on the boundary of a sufficiently large domain.  
For the sake of notational simplicity and clarity of presentation, we restrict our discussion to the one-dimensional case of \eqref{def-Dirac-simple}, for which the numerical schemes and the corresponding analysis are presented in detail. The extension to \eqref{def-Dirac-complete} and higher-dimensional settings is standard and straightforward, and all the results established in this work remain valid without essential modifications.

\noindent
Moreover, in order to avoid unnecessary technicalities, we confine ourselves to the case where the potential functions are independent of time. The more general situation involving explicitly time-dependent potentials can be handled by means of the time-ordering technique; we refer the reader to \cite{yin_sp_time-dependent} for a detailed treatment.

\subsection{Operator splitting for the Dirac equation}
\noindent
For the Dirac equation in 1D $\left( d = 1 \right)$, we consider the following initial-boundary value problem:

\begin{align}
i \epsilon \partial_t \Phi(t,\mathbf{x}) 
&= \left[ \left( -i \epsilon \sigma_1 \partial_x + \sigma_3 \right) 
+ \left( V(x) I_2 - A_1(x) \sigma_1 \right) \right] 
\Phi(t,x) \notag \\
&\qquad (t,x) \in \mathbb{R}^+ \times \Omega, \label{def_Dirac_simplified}\\
\Phi(t, a) &= \Phi(t, b), \quad 
\partial_{x}\Phi(t, a) = \partial_{x}\Phi(t, b), \quad t > 0, \\
\Phi(0,x) &= \Phi_0(x), \quad x \in \bar{\Omega}.
\end{align}

\noindent
where $\Phi, \Phi_0 \in \mathbb{C}^2, x \in \Omega := (a,b), t \in \left[0,T_\text{end}\right]$. To simplify the notation, we define:

\begin{align}
    & T = -\sigma_1 \partial_x - \frac{i}{\epsilon} \sigma_3, \label{def_T}\\
    & W = -\frac{i}{\epsilon} V(x) I_2 + \frac{i}{\epsilon} A_1 (x) \sigma_1. \label{def_W}
\end{align}

\noindent
Then the Dirac equation can be written in the compact form
\begin{equation}\label{def_dirac_operato}
    \partial_t \Phi = (T + W)\Phi.
\end{equation}
Formally, its solution is given by
\begin{equation}
    \Phi(t,x)=e^{t(T+W)}\Phi(0,x).
\end{equation}
Applying the operator splitting technique, the evolution operator can be approximated by
\begin{equation}
    \Phi(t,x)\approx \prod_{j=1}^{n} e^{a_j t T} e^{b_j t W}\Phi(0,x),
\end{equation}
where $n\geq 1$, and the coefficients
\[
a_j,b_j\in\mathbb{R}, \qquad j=1,2,\ldots,n,
\]
are parameters to be determined. The coefficients are chosen such that the resulting splitting scheme attains a prescribed order of accuracy with respect to the time step.

\noindent

\subsection{The full-discretization of the splitting methods}

\noindent
Let $h = \frac{b - a}{M}$ denote the mesh size with $M$ being an even integer and $\tau>0$ denote the time step. We denote the grid points and the time steps as:

\begin{align}
x_j = a + j h,\quad j = 0,1,\dots,M-1; \quad
t_n = n \tau,\quad n = 0,1,2,\dots
\end{align}

\noindent
Let $\Phi_{j}^{ n} \approx \Phi(t_n, x_j)$ be the numerical solution at the grid point $x_j$ and the time step $t_n$, $V_j := V( x_j)$ and $A_{1,j} := A_1(x_j)$ denote the electric and magnetic potentials, respectively.

\noindent
The Fourier interpolation function of $f$ based on the given grid points can be formulated as follows:

\begin{align}
      & f_I(x) = \sum_{l = -\frac{M}{2}}^{\frac{M}{2}-1}\hat{f}_l e^{i \mu_{l}(x - a)}, \quad\hat{f}_l = \frac{1}{M} \sum_{j = 0}^{M - 1} f(x_j) e^{-i \mu_l(x_j - a)}, \label{fourier interpolation function}\\
      & \mu_l = \frac{2 \pi l}{b - a},\quad l = - \frac{M}{2}, \cdots, \frac{M}{2} - 1. \label{def_mu_l}
\end{align}
 
\noindent
In this paper, we define the $l^2$-norm for a function as:

\begin{equation} \label{l^2-norm}
    \left\| \Phi(t, \cdot) \right\|_{L^2} := \sqrt{\int_{\Omega} \left| \Phi(t,x)^2 \right| dx}.
\end{equation}

\subsubsection{The first-order time-splitting method ($S_1$) }
\noindent
In the first-order splitting, from time $t = t_n$ to time $t = t_{n+1}$, we solve the Dirac equation in two sub-steps. The approximate solution of the Dirac equation \eqref{def_dirac_operato} can be computed by the following iterative formula: 

\begin{align}
    & \Phi(t_{n+1}, \cdot) = e^{\tau (W +T) } \Phi(t_{n}, \cdot) \approx e^{\tau  W } e^{\tau  T} \Phi(t_{n}, \cdot), \quad n \geq 0.
\end{align}

\noindent
Specifically, in the first sub-step, we compute $e^{\tau T}\Phi(t_n, \cdot)$ in the phase space using the Fourier spectral method. After that, in the second sub-step, we integrate the resulting system exactly in time for the $W$ operator. More precisely, we denote by $\Phi^n=\left(\Phi_0^n,\Phi_1^n,\cdots,\Phi_{M-1}^n\right)$ the numerical solution at the temporal grid point \(t=t_n\).
The proposed fully discrete scheme is formulated as follows: for each \(n\geq 0\),

\begin{align}
     &\Phi^{n, *}_j = \sum_{l = -\frac{M}{2}}^{\frac{M}{2}-1} e^{\tau \left( -i \sigma_1 \mu_l - \frac{i}{\epsilon} \sigma_3 \right)} \hat{\Phi}^n_l e^{i \mu_l (x_{j} - a)},   \label{sp1_1}\\
     &\Phi^{n+1}_j = e^{\tau  \left[ -\frac{i}{\epsilon} V(  x_j) I_2 + \frac{i}{\epsilon} A_1 ( x_j) \sigma_1 \right] } \Phi^{n, *}_j,   \label{sp1_2}   
\end{align}

\noindent
where $\mu_l$ is defined as the equality \eqref{def_mu_l}, the initial condition is given by $ \Phi^0=\left( \Phi^0_0,\Phi^0_1,\cdots,\Phi^0_{M-1} \right)=\left( \Phi^0(x_0),\Phi^0(x_1),\cdots,\Phi^0(x_{M-1}) \right)$ and $\hat{\Phi}^n_j$ gives the Fourier coefficients of $\Phi_j^n$

\begin{equation} \label{sp1_3}
    \hat{\Phi}^n_j = \frac{1}{M} \sum_{j = 0}^{M-1} \Phi^n_j e^{-i \mu_l(x_j - a)}.
\end{equation}
Obviously, the main computational cost of this algorithm comes from FFT, which scales as $O(M\log M)$.

\subsubsection{The second-order time-splitting method ($S_2$)}

\noindent
To improve the accuracy of the approximate solution, the $S_2$ algorithm is designed to narrow the gap between the exact solution operator and the splitting operator \cite{strang-splitting-1968}. Specifically, from time $t = t_n$ to time $t = t_{n + 1}$, we split the solution operator as follows:

\begin{align}
    & \Phi(t_{n+1}, \cdot) = e^{\tau (W +T) } \Phi(t_{n}, \cdot) \approx e^{\frac{\tau}{2}  W } e^{\tau  T} e^{\frac{\tau}{2}  W } \Phi(t_{n}, \cdot).
\end{align}
Similar to $S_1$, the approximate solution can be obtained by the following steps for $n\geq 0$:

\begin{align}
    & \Phi^{n, *}_j = e^{\frac{1}{2} \tau  \left[ -\frac{i}{\epsilon} V(  x) I_2 + \frac{i}{\epsilon} A_1 (x) \sigma_1 \right] } \Phi^n_j,   \label{sp2_1} \\ 
    & \Phi ^{n,**}_j = \sum_{l = -\frac{M}{2}}^{\frac{M}{2}-1} e^{\tau \left( -i \sigma_1 \mu_l - \frac{i}{\epsilon} \sigma_3 \right)} \hat{\Phi}^{n, *}_l e^{i \mu_l(x_j - a)},   \label{sp2_2} \\
    & \Phi_j ^{n + 1} = e^{\frac{1}{2} \tau  \left[ -\frac{i}{\epsilon} V(  x) I_2 + \frac{i}{\epsilon} A_1 (x) \sigma_1 \right] } \Phi^{n, **}_j,   \label{sp2_3}
\end{align}
\noindent
where the Fourier coefficients of $\Phi^{n, *}_j$ are defined as 

\begin{equation} \label{sp2_4}
    \hat{\Phi}^{n, *}_l = \frac{1}{M} \sum_{j = 0}^{M-1} \Phi^{n, *}_j e^{-i \mu_l (x_j - a)}.
\end{equation}

\noindent
We remark here that in practical computation, $S_2$ can be more effectively implementated through the following observation:

\begin{equation} \label{consolidation_S2}
\begin{aligned}
    e^{T_\text{end}(T + W)}  &\approx e^{\frac{\tau}{2} W} e^{\tau T} e^{\frac{\tau}{2} W} \cdots e^{\frac{\tau}{2} W} e^{\tau T} e^{\frac{\tau}{2} W} \\
    & = e^{\frac{\tau}{2} W} e^{\tau T} e^{\tau W} e^{\tau T} e^{\tau W} \cdots e^{\tau T} e^{\tau W} e^{\frac{\tau}{2} W}.
\end{aligned}
\end{equation}

\noindent
It is apparent that under this framework, the computational cost for $S_2$ is nearly identical to that of $S_1$, yet it cab yield a substantial reduction in the approximation error.

\subsubsection{The compact fourth-order time-splitting method ($S_\text{4c}$)}

\noindent
We can further reduce the error by enhancing the accuracy of the splitting scheme. In addition, to avoid the negative time steps~\cite{bao2019fourth,chin2001fourth}, we use $S_\text{4c}$ to obtain the approximate solution \cite{bao2019fourth} by using the following formula for $n\geq 0$: 

\begin{equation} \label{def_s4c}
    \Phi^{n + 1} = e^{\frac{1}{6} \tau W } e^{\frac{\tau}{2} T} e^{\frac{2}{3} \tau W + \frac{\tau^3}{72} \left[ W, \left[ T, W \right] \right]} e^{\frac{\tau}{2} T} e^{\frac{1}{6} \tau W} \Phi^n.
\end{equation}

\noindent
In this scheme, we need to compute an additional operator 
$e^{\frac{2}{3} \tau W + \frac{\tau^3}{72} \left[ W, \left[T, W \right]\right]}. $
Similar to the computation in \cite{bao2019fourth}, we can find out that this operator contains partial derivatives $\partial_{x_1}, \cdots, \partial_{x_d}$ in 2D and 3D. As a result, in higher dimensions, to evaluate this operator and maintain fourth-order accuracy, we use the fourth-order Runge–Kutta method (RK4) in time and spectral discretization in space. 
We remark here that a simplification similar to \eqref{consolidation_S2} can be applied to $S_\text{4c}$. For the sake of brevity, the detailed algorithm for $S_\text{4c}$ is omitted here.

\subsection{Mass Conservation}
Let $X_M=
\Bigl\{
U=(U_0,U_1,\cdots,U_M)
\ \big|\
U_j\in\mathbb{C}^2,\ j=0,1,\cdots,M,\ U_0=U_M
\Bigr\}$, and adopt the periodic extension $U_{-1}=U_{M-1},
U_{M+1}=U_1$, whenever such indices are involved. On the space $X_M$, the discrete $l_2$-norm is defined by
\begin{equation}
\|U\|_{l_2}
=
\sqrt{h\sum_{j=0}^{M-1}|U_j|^2},
\qquad
U\in X_M.
\end{equation}

\noindent
 We establish the discrete mass conservation property for the time-splitting scheme $S_1$. 

\begin{lemma} \label{lemma_sp1_mass_conservation}
The time-splitting scheme $S_1$ \eqref{sp1_1} and \eqref{sp1_2} conserves the discretized mass under any mesh size $h$ and temporal step size $\tau$, i.e. 
    
\begin{align}
    \left\| \Phi^{n + 1} \right\|_{l_2}^2 &= h \sum_{j = 0}^{M-1} \left| \Phi^{n + 1}_j \right|^2 \equiv h \sum_{j = 0}^{M-1} \left| \Phi^{0}_j \right|^2 \notag\\ &= h \sum_{j = 0}^{M-1} \left| \Phi_0 (x_j) \right|^2 = \left\| \Phi^0 \right\|_{l_2}^2 .
\end{align}

\end{lemma}

\begin{proof}
Using the unitarity of $e^{\tau\left(
-\frac{i}{\epsilon}V(x_j)I_2
+\frac{i}{\epsilon}A_1(x_j)\sigma_1
\right)}$ , we obtain
\begin{equation*}
\begin{aligned}
\frac{1}{h}\|\Phi^{n+1}\|_{l_2}^2
&=
\sum_{j=0}^{M-1} |\Phi_j^{n+1}|^2 =
\sum_{j=0}^{M-1}
\left|
e^{\tau\left(
-\frac{i}{\epsilon}V(x_j)I_2
+\frac{i}{\epsilon}A_1(x_j)\sigma_1
\right)}
\Phi_j^{n,*}
\right|^2 \\
&=
\sum_{j=0}^{M-1} |\Phi_j^{n,*}|^2 .
\end{aligned}
\end{equation*}

\noindent
On the other hand, by the orthogonality of the discrete Fourier basis,
\begin{equation*}
\begin{aligned}
\sum_{j=0}^{M-1} |\Phi_j^{n,*}|^2
&=
\sum_{j=0}^{M-1}
\left|
\sum_{l=-\frac{M}{2}}^{\frac{M}{2}-1}
e^{\tau\left(
-i\sigma_1\mu_l-\frac{i}{\epsilon}\sigma_3
\right)}
\hat{\Phi}_l^n
e^{i\mu_l(x_j-a)}
\right|^2 \\
&=
M
\sum_{l=-\frac{M}{2}}^{\frac{M}{2}-1}
\left|
e^{\tau\left(
-i\sigma_1\mu_l-\frac{i}{\epsilon}\sigma_3
\right)}
\hat{\Phi}_l^n
\right|^2 \\
&= M \sum_{l=-\frac{M}{2}}^{\frac{M}{2}-1}
|\hat{\Phi}_l^n|^2 = M \sum_{l = -\frac{M}{2}}^{\frac{M}{2}-1} \left | \frac{1}{M} \sum_{j = 0}^{M-1} \Phi^n_j e^{-i \mu_l(x_j - a)} \right| ^2\\ &= \sum_{l = -\frac{M}{2}}^{\frac{M}{2}-1} \left | \sum_{j = 0}^{M-1} \Phi^n_j e^{-i \mu_l(x_j - a)} \right| ^2 = \sum_{j = 0}^{M-1} \left| \Phi^n_j \right|^2 = \frac{1}{h} \left \| \Phi^{n} \right\|_{l_2} ^2.
\end{aligned}
\end{equation*}
Therefore,
\begin{equation*}
\|\Phi^{n+1}\|_{l_2}
=
\|\Phi^n\|_{l_2}.
\end{equation*}
The desired result follows immediately by induction on $n$.     \hfill $\square$
\end{proof}

\noindent
Similar conservation properties hold for $S_2$, as well as $S_\text{4c}$ in 1D. For brevity, we only give the mass conservation result and omit the proof here.
\begin{lemma}\label{lemma_sp2_mass_conservation}
Both the second-order splitting scheme \(S_2\)
\eqref{sp2_1}--\eqref{sp2_3} and the fourth-order compact
splitting scheme \(S_{4c}\) \eqref{def_s4c} in 1D
conserve the discrete mass, namely,
\begin{equation}
\|\Phi^{n+1}\|_{l_2}
=
\|\Phi^n\|_{l_2}
=
\|\Phi^0\|_{l_2},
\qquad n\geq 0.
\end{equation}
Equivalently,
\begin{align}
\|\Phi^{n+1}\|_{l_2}^2 = h\sum_{j=0}^{M-1} |\Phi_j^{\,n+1}|^2 = h\sum_{j=0}^{M-1} |\Phi_0(x_j)|^2 .
\end{align}
\end{lemma}

\section{Error Estimates} \label{section3}

\noindent
For simplicity, in this paper we only show the error estimates of \eqref{def-Dirac-simple} for the case $d = 1$. The generalization to the Dirac equation \eqref{def-Dirac-complete} and higher dimensions will follow directly with few modifications. 

\noindent
To obtain a proper error estimate, our analysis assumes that the exact solution and the electromagnetic potentials satisfy certain regularity conditions:

\begin{equation}
\begin{aligned} 
  (A) \qquad  &\Phi \in C^2 \left( \left[0,T_\text{end}\right]; W^{m,\infty}_p \left(\Omega \right) \right), \quad V, A_1 \in C^m\left(\Omega \right), \quad m \geq 4,
\end{aligned}
\end{equation}
\noindent
where $W^{m,\infty}_p \left(\Omega \right) = \left\{ f | f \in W^{m,\infty} \left( \Omega \right), \partial_x^k f(a) = \partial_x^k f(b), k = 0,\cdots,m - 1 \right\}$, and $m$ can be any integer with $m \geq 4$, which will be related to the numerical accuracy in space.
In practice, the electromagnetic potentials $V$ and $A_1$ may not be inherently periodic, but they nearly attain constant values at the boundaries of $\Omega$. Under this circumstance, they can be reasonably treated as periodic over $\Omega$. Consequently, the wave function $\Phi$ will also be periodic.

\noindent
In addition, the derivatives of the wave function and the electromagnetic potentials should satisfy the following conditions:

\begin{equation}
\begin{aligned}
 (B) \qquad  & \left \| \frac{\partial^{r + k}}{\partial t^{r} \partial x^k} \Phi \right \|_{L^\infty([0,T_{end}]; (L^\infty(\Omega))^2)} \leq \frac{C_{r + k}}{\epsilon^{r + k}}\\
 &\left \| \frac{\partial^{ k}}{\partial x^k} V  \right \|_{L^\infty((L^\infty(\Omega))^2)} \leq D_{k}, \quad\left \| \frac{\partial^{k}}{\partial x^{k}} 
   A_1 \right \|_{L^\infty((L^\infty(\Omega))^2)} \leq D_{k}, \\
   &  \qquad \qquad \qquad0 \leq r \leq 2,\, 0 \leq k \leq m.
\end{aligned}
\end{equation}
\noindent
where the constants $C_{r + k}$ and $ D_{k}$ are independent of $\epsilon$. 

\subsection{The Main Results}
\noindent
Based on the assumptions above, we can get the main results of the error estimates. To simplify the notation, we denote the numerical interpolation function of $\Phi^n$ by $\Phi_I^n$ at time $t = t_n$. Specifically, we have:
\begin{equation}\label{def_numerical_interpolation}
    \Phi^n_I(x) = \sum_{l = -\frac{M}{2}}^{\frac{M}{2}-1}  \hat{\Phi}^n_l e^{i \mu_l(x - a)},
     \quad \hat{\Phi}^n_l = \frac{1}{M} \sum_{j = 0}^{M-1} \Phi^n_j e^{-i \mu_l (x - a)},
\end{equation}
with $\mu_l = \frac{2 \pi l}{b - a}$, $l=-\frac{M}{2}, -\frac{M}{2}+1, \cdots, \frac{M}{2}-1$. Then we can get the following theorems for the error estimates.

\begin{thm} \label{thm-error-fun-sp1}
Let $\Phi(t_n, \cdot)$ be the exact solution of \eqref{def_Dirac_simplified} at time $t = t_n$ and $\Phi_I^{n}(\cdot)$ be the numerical solution given by the $S_1$ method. Under assumptions (A) and (B), and assuming $ 0< \epsilon \leq 1$, $\tau = o(\epsilon)$, $h = o(\epsilon)$, then the following error estimate holds for $S_1$ with $0 \leq n \leq {T_\text{end}}/{\tau}$:

    \begin{align} \label{equ-thm-error-fun-sp1}
        &\left\| \Phi(t_n,\cdot) - \Phi_I^{n}(\cdot) \right\|_{L^2} \leq C_1 T_\text{end} \left( \frac{\tau}{\epsilon^2} \right)  + E_{m,1} \cdot \frac{T_\text{end}}{\tau} \left( \frac{h}{\epsilon(b - a)} \right)^m,
    \end{align}
    \noindent
    where $C_1$ and $E_{m,1}$ are positive constants independent of $\epsilon, \tau$ and $h$. 
    
\end{thm}

\begin{thm} \label{thm-error-fun-sp2}
 Let $\Phi(t_n,\cdot)$ be the exact solution of \eqref{def_Dirac_simplified} at time $t = t_n$ and $\Phi_I^{n}(\cdot)$ be the numerical solution given by the $S_2$ method. Under assumptions (A) and (B), and assuming $ 0< \epsilon \leq 1$, $\tau = o(\epsilon)$, $h = o(\epsilon)$, the following error estimate for $S_2$ holds for $0 \leq n \leq T_\text{end} /\tau$

    \begin{equation} \label{equ-thm-error-fun-sp2}
        \left\| \Phi(t_n,\cdot) - \Phi_I^{n} \right\|_{L^2} \leq C_2 T_\text{end} \left( \frac{\tau^2}{\epsilon^3} \right) + E_{m,2} \cdot \frac{T_\text{end}}{\tau} \left( \frac{h}{\epsilon(b - a)} \right)^m,
    \end{equation}
    \noindent
    where $C_2$ and $E_{m,2}$ are  positive constants independent of $\epsilon, \tau, h$. 
\end{thm}

\begin{thm} \label{thm-error-fun-S4c}
Let $\Phi(t_n,\cdot)$ be the exact solution of \eqref{def_Dirac_simplified} at time $t = t_n$ and $\Phi_I^{n}(\cdot)$ be the numerical solution given by the $S_\text{4c}$ method. Under assumptions (A) and (B), and assuming $ 0< \epsilon \leq 1$, $\tau = o(\epsilon)$, $h = o(\epsilon)$, the following error estimate for $S_\text{4c}$ in 1D holds for $0 \leq n \leq T_\text{end} / \tau$

    \begin{equation} \label{equ-thm-error-fun-s4c}
        \left\| \Phi(t_n) - \Phi_I^{n} \right\|_{L^2} \leq C_3 T_\text{end} \left( \frac{\tau^4}{\epsilon^5}  \right) + E_{m,3} \cdot \frac{T_\text{end}}{\tau} \left( \frac{h}{\epsilon(b - a)} \right)^m,
    \end{equation}
    \noindent
    where $C_3$ and $E_{m,3}$ are positive constants independent of $\epsilon, \tau, h$. 
\end{thm}

\noindent
Furthermore, similar error estimates for the observables, such as the total probability density and the current density can be established by using the control relation between error estimates for the observables and the error estimate for the wave function $\Phi$ \cite{sp_SchrodingerEquation}, where the total probability density $\rho$ and the current density $J$ are defined as follows:

\begin{align}
    & \rho(t,x) = \Phi(t,x)^* \Phi(t,x), \label{def_total_probability_density}\\
    & J(t,x) = \Phi(t,x)^* \sigma_1 \Phi(t,x). \label{def_ current density} 
\end{align}
Similar to $\Phi_I^n(x)$, we can define the numerical interpolation functions $\rho_I^n(\cdot)$ and $J_I^n(\cdot)$:
\begin{align}
    &\rho^n_I(x) = \sum_{l = -\frac{M}{2}}^{\frac{M}{2}-1}  \hat{\rho}^n_l e^{i \mu_l(x - a)},
     \quad \hat{\rho}^n_l = \frac{1}{M} \sum_{j = 0}^{M-1} \rho^n_j e^{-i \mu_l (x - a)},\\
     & J^n_I(x) = \sum_{l = -\frac{M}{2}}^{\frac{M}{2}-1}  \hat{J}^n_l e^{i \mu_l(x - a)},
     \quad \hat{J}^n_l = \frac{1}{M} \sum_{j = 0}^{M-1} J^n_j e^{-i \mu_l (x - a)}.
\end{align}

\begin{corollary} \label{corollary-sp1}
    Let $\rho (t_n,\cdot)$ and $J(t_n,\cdot)$ be the exact total probability density and current density of \eqref{def_Dirac_simplified} at time $t = t_n$ and $\rho_I^n(\cdot)$ and $J_I^n(\cdot)$ be the numerical total probability density and current density given by the $S_{1}$ method. Under assumptions (A) and (B), and assuming $ 0< \epsilon \leq 1$, $\tau = o(\epsilon)$, $h = o(\epsilon)$, the following error estimates for the $S_1$ method hold for $0 \leq n \leq T_\text{end} / \tau$
    
\begin{align}
    &\left\| \rho(t_n,\cdot) - \rho_I^{n}(\cdot) \right\|_{L^2} \leq \hat{C}_1 T_\text{end} \left( \frac{\tau}{\epsilon^2} \right) + \hat{E}_{m,1} \cdot \frac{T_\text{end}}{\tau} \left( \frac{h}{\epsilon(b - a)} \right)^m \label{eq1-corollary-sp1}, \\
    &\left\| J(t_n,\cdot) - J_I^{n}(\cdot) \right\|_{L^2} \leq \hat{C}_1 T_\text{end} \left( \frac{\tau}{\epsilon^2} \right) + \hat{E}_{m,1} \cdot \frac{T_\text{end}}{\tau} \left( \frac{h}{\epsilon(b - a)} \right)^m, \label{eq2-corollary-sp1}
\end{align}
\noindent
where $\hat{C}_1$ and $\hat{E}_{m,1}$ are positive constants independent of $\epsilon, \tau, h$.
\end{corollary}

\begin{corollary} \label{corollary-sp2}
Let $\rho (t_n,\cdot)$ and $J(t_n,\cdot)$ be the exact total probability density and current density of \eqref{def_Dirac_simplified} at time $t = t_n$ and $\rho_I^n(\cdot)$ and $J_I^n(\cdot)$ be the numerical total probability density and current density given by the $S_2$ method. Under assumptions (A) and (B), and assuming $ 0< \epsilon \leq 1$, $\tau = o(\epsilon)$, $h = o(\epsilon)$, the following error estimates for the $S_2$ method hold for $0 \leq n \leq T_\text{end} / \tau$
    
\begin{align}
    &\left\| \rho(t_n,\cdot) - \rho_I^{n}(\cdot) \right\|_{L^2} \leq \hat{C}_2 T_\text{end} \left( \frac{\tau^2}{\epsilon^3} \right) + \hat{E}_{m,2} \cdot \frac{T_\text{end}}{\tau} \left( \frac{h}{\epsilon(b - a)} \right)^m \label{eq1-corollary-sp2}, \\
    &\left\| J(t_n,\cdot) - J_I^{n}(\cdot) \right\|_{L^2} \leq \hat{C}_2 T_\text{end} \left( \frac{\tau^2}{\epsilon^3} \right) + \hat{E}_{m,2} \cdot \frac{T_\text{end}}{\tau} \left( \frac{h}{\epsilon(b - a)} \right)^m, \label{eq2-corollary-sp2}
\end{align}
\noindent
where $\hat{C}_2$ and $\hat{E}_{m,2}$ are positive constants independent of $\epsilon, \tau, h$.
\end{corollary}

\begin{corollary} \label{corollary-s4c}
    Let $\rho (t_n,\cdot)$ and $J(t_n,\cdot)$ be the exact total probability density and current density of \eqref{def_Dirac_simplified} at time $t = t_n$ and $\rho_I^n(\cdot)$ and $J_I^n(\cdot)$ be the numerical total probability density and current density given by the $S_{4c}$ method. Under assumptions (A) and (B), and assuming $ 0< \epsilon \leq 1$, $\tau = o(\epsilon)$, $h = o(\epsilon)$, the following error estimates for the $S_{4c}$ method hold for $0 \leq n \leq T_\text{end} / \tau$
    
\begin{align}
    &\left\| \rho(t_n,\cdot) - \rho_I^{n}(\cdot) \right\|_{L^2} \leq \hat{C}_3 T_\text{end} \left( \frac{\tau^4}{\epsilon^5} \right) + \hat{E}_{m,3} \cdot \frac{T_\text{end}}{\tau} \left( \frac{h}{\epsilon(b - a)} \right)^m \label{eq1-corollary-s4c}, \\
    &\left\| J(t_n,\cdot) - J_I^{n}(\cdot) \right\|_{L^2} \leq \hat{C}_3 T_\text{end} \left( \frac{\tau^4}{\epsilon^5} \right) + \hat{E}_{m,3} \cdot \frac{T_\text{end}}{\tau} \left( \frac{h}{\epsilon(b - a)} \right)^m, \label{eq2-corollary-s4c}
\end{align}
\noindent
where $\hat{C}_3$ and $\hat{E}_{m,3}$ are positive constants independent of $\epsilon, \tau, h$.
\end{corollary}

\noindent
From the above theorems and corollaries, we can obtain that if $\delta$ is the desired error bound, the associated discretization parameters shall satisfy the following conditions:

\begin{align}
    & S_1: \frac{\tau}{\epsilon^2} \lesssim \delta , \qquad \frac{h}{\epsilon} \lesssim \left(\delta \tau\right) ^{\frac{1}{m}}, \label{accuracy-sp1} \\
    & S_2: \frac{\tau^2}{\epsilon^3} \lesssim \delta , \qquad \frac{h}{\epsilon} \lesssim \left(\delta \tau\right) ^{\frac{1}{m}}, \label{accuracy-sp2} \\
    & S_\text{4c}: \frac{\tau^4}{\epsilon^5} \lesssim \delta , \qquad \frac{h}{\epsilon} \lesssim \left(\delta \tau\right) ^{\frac{1}{m}}. \label{accuracy-s4c}
\end{align}

\begin{remark}
    The above conditions show the $\epsilon$-scalability of $\tau$ and $h$. From these results, it is clear that the requirements $\tau = o(\epsilon), h = o(\epsilon)$ in Theorems \ref{thm-error-fun-sp1} - \ref{thm-error-fun-S4c} are reasonable.
\end{remark}

\subsection{Proof of the Main Results}

\noindent
Before proving the Theorems \ref{thm-error-fun-sp1} - \ref{thm-error-fun-S4c}, we first prove the following lemmas.

\begin{lemma} [Baker--Campbell--Hausdorff formula] \label{Baker--Campbell--Hausdorff expansion}
Let $X,Y$ be elements of a Banach algebra with norm $\|\cdot\|$. Then
\begin{equation}
\log(e^X e^Y)
= X + Y + \frac{1}{2}[X,Y]
+ \frac{1}{12}[X,[X,Y]]
- \frac{1}{12}[Y,[X,Y]]
+ R_4(X,Y),
\end{equation}
\noindent
where $R_4(X,Y)$ consists of iterated commutators of order at least $4$.
Moreover, there exists a constant $C_4>0$ such that
\begin{equation}
\|R_4(X,Y)\| \le C_4(\|X\|+\|Y\|)^4 .
\end{equation}
\end{lemma}

\begin{proof}
    We refer to \cite{hall2015baker} for the proof.     \hfill $\square$
\end{proof}

\begin{lemma} \label{lemma-splitting-error}
    Under the assumptions of Theorems \ref{thm-error-fun-sp1}, the following estimates for the time-splitting operators hold:
    
    \begin{align} 
        & \left\| \left[ e^{\tau \left( T + W \right)} - e^{\tau W} e^{\tau T} \right] \Phi(t_n, \cdot) \right\|_{L^2} = O \left(\frac{\tau^2}{\epsilon^2} \right),  \label{eq1-lemma-splitting-error} \\
        &\left\| \left[ e^{\tau \left( T + W \right)} - e^{\frac{\tau}{2} W} e^{\tau T} e^{\frac{\tau}{2} W} \right] \Phi(t_n, \cdot) \right\|_{L^2} = O \left(\frac{\tau^3}{\epsilon^3} \right),  \label{eq2-lemma-splitting-error} \\
        &\left\| \left[ e^{\tau \left( T + W \right)} - e^{\frac{1}{6} \tau W } e^{\frac{\tau}{2} T} e^{\frac{2}{3} \tau W + \frac{\tau^3}{72} \left[ W, \left[ T, W \right] \right]} e^{\frac{\tau}{2} T} e^{\frac{1}{6} \tau W} \right] \Phi(t_n, \cdot) \right\|_{L^2}\notag\\& \qquad = O \left(\frac{\tau^5}{\epsilon^5} \right),\label{eq3-lemma-splitting-error}
    \end{align}
\noindent
where $T$ and $W$ are defined in \eqref{def_T}-\eqref{def_W}, and $\Phi(t_n, \cdot)$ is the exact solution at time $t = t_n$.
\end{lemma}

\begin{proof}

\noindent
    By Lemma \ref{Baker--Campbell--Hausdorff expansion} and the Taylor expansions of the exponential operators, it suffices to compute the commutators on $\Phi(t_n, \cdot)$.
    To simplify the notation, We define: 
    
    \begin{equation}
    R_n(x) = \tau \left( -\frac{i}{\epsilon} V(x) I_2  + \frac{i}{\epsilon} A_1(x) \sigma_1 \right),
    \end{equation}
\noindent
We can compute the commutator $[T, W]$ acting on $\Phi(t_n, \cdot)$ as follows:
\begin{equation}
    \begin{aligned}
    &\left[ T, W \right] \Phi(t_n, \cdot) \\
    &= \tau R_{n}(x) \left( -\sigma_1 \partial_x \Phi(t_n, \cdot) - \frac{i}{\epsilon} \sigma_3 \Phi(t_n, \cdot) \right) + \tau \left( \sigma_1 \partial_x + \frac{i}{\epsilon} \sigma_3 \right) R_{n}(x) \Phi(t_n, \cdot) \\
    &=  - \frac{i}{\epsilon} \tau R_{n}(x) \sigma_3 \Phi(t_n, \cdot) + \tau \sigma_1 R_{n}'(x) \Phi(t_n, \cdot) + \frac{i}{\epsilon} \tau \sigma_3 R_{n}(x) \Phi(t_n, \cdot) \\
    &= \tau \left[ \sigma_1 R_{n}'(x) \Phi(t_n, \cdot) - \frac{2 \tau}{\epsilon^2}   A_1( x) \sigma_3 \sigma_1 \, \Phi(t_n, \cdot) \right].
    \end{aligned}
\end{equation}

\noindent
Similarly, we can derive $[T, [T, W]]\Phi$ by:
\begin{equation}
\begin{aligned}
& \frac{1}{\tau^2} \left[ T, \left[ T, W \right] \right] \Phi(t_n, \cdot) = \frac{1}{\tau^2} \left( T^2 W - 2T WT + W T^2 \right) \Phi(t_n, \cdot) \\
&= \Bigg[ \left( \sigma_1 \, \partial_x + \frac{i}{\epsilon} \sigma_3 \right)^2 R_n - 2 \left( \sigma_1 \, \partial_x + \frac{i}{\epsilon} \sigma_3 \right) R_n \left( \sigma_1 \, \partial_x + \frac{i}{\epsilon} \sigma_3 \right) \notag\\& \qquad+ R_n \left( \sigma_1 \, \partial_x + \frac{i}{\epsilon} \sigma_3 \right)^2 \Bigg] \Phi(t_n, \cdot) \\
&= \Biggl[ \left( \partial_{xx} - \frac{1}{\epsilon^2} \right) R_n \Phi(t_n, \cdot)
- 2 \left( \sigma_1 \partial_x + \frac{i}{\epsilon} \sigma_3  \right) R_n
\left( \sigma_1 \partial_x \Phi(t_n, \cdot) + \frac{i}{\epsilon} \sigma_3 \Phi(t_n, \cdot) \right)\\
& \qquad+ R_n \left( \partial_{xx}\Phi(t_n, \cdot) - \frac{1}{\epsilon^2} \Phi(t_n, \cdot) \right) \Biggr] \\
&= \left[
R_n'' - \frac{2}{\epsilon^2} R_n
- \frac{2i}{\epsilon} \sigma_1 R_n' \sigma_3
+ \frac{2}{\epsilon^2} \sigma_3 R_n \sigma_3
\right] \Phi(t_n, \cdot)
+ \left[ - \frac{2i}{\epsilon} \sigma_1 R_n \sigma_3 - \frac{2i}{\epsilon} \sigma_3 R_n \sigma_1
\right] \partial_x \Phi(t_n, \cdot) \\
&=\left[\left(R_n'' - \frac{2i}{\epsilon} \sigma_1 R_n' \sigma_3 - \frac{4i \tau}{\epsilon^3} A_1(x) \sigma_1  \right)\Phi(t_n, \cdot) + \frac{4 \tau}{\epsilon^2} A_1(x) \sigma_3 \partial_x \Phi(t_n, \cdot)
\right].
\end{aligned}
\end{equation}
\noindent
Therefore, $ \left\| \left[ T, W \right] \Phi(t_n, \cdot) \right\|_{L^2} = O \left(\frac{\tau^2}{\epsilon^2}\right)$ and $\left\| \left[ T, \left[ T, W \right] \right] \Phi(t_n, \cdot) \right\|_{L^2} = O \left( \frac{\tau^3}{\epsilon^3} \right)$.

\noindent
Similarly, it holds that
$\left\| \left[ W,\left[ T,W \right]\right]\Phi(t_n,\cdot) \right\|_{L^2}
= O\left(\frac{\tau^3}{\varepsilon^3}\right)$.
Therefore, invoking Lemma~\ref{Baker--Campbell--Hausdorff expansion},
the estimates \eqref{eq1-lemma-splitting-error} and \eqref{eq2-lemma-splitting-error} can be established through the corresponding third-order commutator analysis. For the estimate \eqref{eq3-lemma-splitting-error}, the Taylor expansion shows that all commutators up to fourth order vanish. Therefore, it remains to compute the fifth commutators, i.e. we only need to prove that $[T,[T,[T, [T, W]]]]\Phi(t_n, \cdot) $ gives rise to a fifth-order error that remains uncanceled by other commutators. Since

\begin{equation}
\begin{aligned}
& [T,[T,[T, [T, W]]]] \\
&= [T,[T,-\frac{4i \tau^3}{\epsilon^3}  A_1(x) \sigma_1 + \frac{4 \tau^3}{\epsilon^2}A_1(x) \sigma_3 \partial_x + o \left(\frac{\tau^3}{\epsilon^3} \right)]] \\
&= [T,[T,-\frac{4i \tau^3}{\epsilon^3} A_1(x) \sigma_1]] + [T,[T, \frac{4 \tau^3}{\epsilon^2}A_1(x) \sigma_3 \partial_x ]] + o \left(\frac{\tau^5}{\epsilon^5} \right).
\end{aligned}
\end{equation}

\noindent
we only need to consider the second term $\left[T,\left[T, \frac{4 \tau^3}{\epsilon^2}A_1(x) \sigma_3 \partial_x \right]\right]$.
Define $ c(x) = A_1(x) \sigma_3 \partial_x $, then 

\begin{equation}
\begin{aligned}
& \left[T,\left[T, \frac{4 \tau^3}{\epsilon^2}A_1(x) \sigma_3 \partial_x \right]\right] \\
&= \frac{4 \tau^5}{\epsilon^2} \left[\left( -\sigma_1 \, \partial_x - \frac{i}{\epsilon} \sigma_3 \right), \left[\left( -\sigma_1 \, \partial_x - \frac{i}{\epsilon} \sigma_3 \right), c(x) \right]\right] \\
&= \frac{4 \tau^5}{\epsilon^2} \left[\left( \sigma_1 \ \partial_x + \frac{i}{\epsilon} \sigma_3 \right), \partial_x A_1(x) \sigma_1 \sigma_3 + 3 A_1(x) \sigma_1 \sigma_3 \partial_{xx} \right] \\
&= \frac{4 \tau^5}{\epsilon^2}
\Bigl[
\partial_{xx} A_1(x)\sigma_3
- \frac{2i}{\epsilon}\partial_x A_1(x)\sigma_1
+ 2\partial_x A_1(x)\sigma_3 \partial_x
\\
&\qquad\qquad
+ 3\partial_x A_1(x)\sigma_3 \partial_{xx}
- \frac{6i}{\epsilon} A_1(x)\sigma_1 \partial_{xx}
+ 6A_1(x)\sigma_3 \partial_{xxx}
\Bigr].
\end{aligned}
\end{equation}
\noindent
Since $\partial_{xxx}$ occurs exclusively in this commutator, and $\left\| \frac{24 \tau^5}{\epsilon^2} A_1(x) \sigma_3 \partial_{xxx} \Phi(t_n, \cdot) \right\|_{L^2}  = O \left( \frac{\tau^5}{\epsilon^5} \right)$, $[T,[T,[T, [T, W]]]]\Phi(t_n, \cdot) $ gives rise to a fifth-order error that remains uncanceled by other commutators. Therefore, \eqref{eq3-lemma-splitting-error} holds.     \hfill $\square$

\end{proof}

\noindent 
Next, we introduce an error estimate for the Fourier interpolation.

\begin{lemma} \label{lemma-fourier-error}
    For a sufficiently smooth function $f(x) \in W^{k,\infty}_p(\Omega)$, $k \geq 1$, let $f_I$ be the Fourier interpolant of $f$ with the grid parameter $M$, then the approximation error between $f$ and $f_I$ is given by:
    \begin{align*}
        \left\| f - f_I \right\|_{L^2} \leq \frac{\bar{C}_k}{M^k}  \left\| \partial^k_x f \right\|_{L^2},
    \end{align*}
    \noindent
    where $\| \cdot\|_{L^2}$ is the standard $L^2$-norm of space $L^2(\Omega)$ and $\bar{C}_k \geq 0$ is a constant depending on $f$ and $k$.
\end{lemma}

\noindent
We refer the proof of the above Lemma \ref{lemma-fourier-error} to \cite{spectral-error} and the details are omitted here for brevity.\\

\noindent
We are now ready to prove Theorem \ref{thm-error-fun-sp1}. The proof of Theorems~\ref{thm-error-fun-sp2} and~\ref{thm-error-fun-S4c} are similar, and thus omitted here.\\

\noindent
\textit{Proof for Theorem \ref{thm-error-fun-sp1}}

\noindent
For convenience, we introduce the auxiliary functions
\[
\Gamma_n(\cdot)=e^{\tau W}e^{\tau T}\Phi(t_n,\cdot),
\qquad
\gamma_n(\cdot)=e^{\tau T}\Phi(t_n,\cdot).
\]
Here, $\Gamma_n$ denotes one exact splitting step starting from the exact solution at time $t_n$. 
Consequently, the difference between $\Phi(t_{n+1},\cdot)$ and $\Gamma_n(\cdot)$ measures the local error generated by the time-splitting procedure. 
By Lemma~\ref{lemma-splitting-error}, this local truncation error satisfies
\begin{equation}
\label{fisrt_term_thm_proof}
    \|\Phi(t_{n+1},\cdot)-\Gamma_n(\cdot)\|_{L^2}
    \le
    C_1\frac{\tau^2}{\epsilon^2}.
\end{equation}

\noindent
To derive the full discretization error estimate, we decompose the error into three different parts corresponding respectively to the temporal splitting error, the Fourier interpolation error, and the propagated numerical error. 
More precisely, by the triangle inequality,
\begin{align}
\label{error_decomposition}
    \|\Phi(t_{n+1},\cdot)-\Phi_I^{n+1}(\cdot)\|_{L^2}
    \le &
    \|\Phi(t_{n+1},\cdot)-\Gamma_n(\cdot)\|_{L^2}
    \nonumber\\
    &
    +\|\Gamma_n(\cdot)-\Gamma_{n,I}(\cdot)\|_{L^2}
    +\|\Gamma_{n,I}(\cdot)-\Phi_I^{n+1}(\cdot)\|_{L^2}.
\end{align}

\noindent
We first estimate the interpolation error term. 
Applying Lemma~\ref{lemma-fourier-error}, we have
\[
    \|\Gamma_n(\cdot)-\Gamma_{n,I}(\cdot)\|_{L^2}
    \lesssim
    M^{-m}\|\partial_x^m\Gamma_n(\cdot)\|_{L^2}.
\]
Therefore, it remains to control the $m$-th order spatial derivative of $\Gamma_n$. 
Recalling the definition of $\Gamma_n$ and using Leibniz's rule, we obtain
\begin{align*}
    \left\|\partial_x^m\Gamma_n(\cdot) \right\|_{L^2}
    &=
    \left\|
    \partial_x^m
    \left(
    e^{\tau W}e^{\tau T}\Phi(t_n,\cdot)
    \right)
    \right\|_{L^2}
    \\
    &\le
    \sum_{q=0}^m
    \binom{m}{q}
    \left\|\partial_x^{m-q}e^{\tau W}\right\|_{L^\infty}
    \,
    \left\|\partial_x^q(e^{\tau T}\Phi(t_n,\cdot))\right\|_{L^2}.
\end{align*}
where $\left\|\partial_x^{m-q}e^{\tau W}\right\|_{L^\infty}$ represents the standard matrix $L^{\infty}$-norm of $\partial_x^{m-q}e^{\tau W}$.
Under the regularity assumptions imposed on $W$ and the exact solution $\Phi$, each term on the right-hand side is bounded, which leads to
\[
    \|\partial_x^m\Gamma_n(\cdot)\|_{L^2}
    \lesssim
    \epsilon^{-m}.
\]
Substituting the above estimate into the Fourier interpolation inequality yields
\begin{equation}
\label{second_term_thm_proof}
    \|\Gamma_n(\cdot)-\Gamma_{n,I}(\cdot)\|_{L^2}
    \lesssim
    \left(
    \frac{h}{\epsilon(b-a)}
    \right)^m.
\end{equation}

\noindent
Next, we turn to the last term in \eqref{error_decomposition}, which corresponds to the propagation of the numerical error from the previous time level. 
To simplify the notation, define
\[
\Gamma^n=
(\Gamma_n(x_0),\ldots,\Gamma_n(x_{M-1})),
\qquad
\gamma^n=
(\gamma_n(x_0),\ldots,\gamma_n(x_{M-1})),
\]
and
\[
\Phi_{\mathrm{exact}}^n=
(\Phi(t_n,x_0),\ldots,\Phi(t_n,x_{M-1})).
\]
Using the definition of the numerical scheme together with Parseval's identity, we derive
\begin{align}
\label{third_term_thm_proof}
    \|\Gamma_{n,I}(\cdot)-\Phi_I^{n+1}(\cdot)\|_{L^2}
    &=
    \|\Gamma^n-\Phi^{n+1}\|_{l^2}
    \nonumber\\
    &=
    \|\gamma^n-\Phi^{n,*}\|_{l^2}
    =
    \|\Phi_{\mathrm{exact}}^n-\Phi^n\|_{l^2}
    \nonumber\\
    &=
    \|\Phi_I(t_n,\cdot)-\Phi_I^n(\cdot)\|_{L^2}.
\end{align}

\noindent
The above identity shows that the error at time $t = t_{n+1}$ can be controlled by the numerical error at the previous step. 
Applying the triangle inequality once again gives
\[
    \|\Phi_I(t_n,\cdot)-\Phi_I^n(\cdot)\|_{L^2}
    \le
    \|\Phi_I(t_n,\cdot)-\Phi(t_n,\cdot)\|_{L^2}
    +
    \|\Phi(t_n,\cdot)-\Phi_I^n(\cdot)\|_{L^2}.
\]
Moreover, Lemma~\ref{lemma-fourier-error} implies that
\[
    \|\Phi_I(t_n,\cdot)-\Phi(t_n,\cdot)\|_{L^2}
    \lesssim
    \left(
    \frac{h}{\epsilon(b-a)}
    \right)^m.
\]

\noindent
Combining the estimates
\eqref{fisrt_term_thm_proof},
\eqref{second_term_thm_proof},
and
\eqref{third_term_thm_proof},
we arrive at the recursive inequality
\[
    \|\Phi(t_{n+1},\cdot)-\Phi_I^{n+1}(\cdot)\|_{L^2}
    \le
    C_1\frac{\tau^2}{\epsilon^2}
    +
    E_{m,1}
    \left(
    \frac{h}{\epsilon(b-a)}
    \right)^m
    +
    \|\Phi(t_n,\cdot)-\Phi_I^n(\cdot)\|_{L^2}.
\]
Finally, applying induction with respect to $n$ immediately yields the desired estimate \eqref{equ-thm-error-fun-sp1}.     \hfill $\square$

\noindent
Since the control coefficients of the splitting error and the Fourier interpolation error are independent of $\epsilon$, the control coefficients of \eqref{equ-thm-error-fun-sp1},\eqref{equ-thm-error-fun-sp2},\eqref{equ-thm-error-fun-s4c} are also independent of $\epsilon$.

\section{Numerical Results} \label{section4}

\noindent
In this section, we conduct the numerical experiments of the time-splitting spectral methods for the semiclassical Dirac equation to verify the error estimates given in Section~\ref{section3}. To evaluate the numerical errors, we use the following relative errors of the wave function $\Phi$, the total probability density function $\rho$ and the current density function $\textbf{J}$: 

\begin{align}
    & e_{\Phi}^r (t_n) = \frac{\left\| \Phi(t_n, \cdot) - \Phi_I^n(\cdot) \right\|_{l^2}}{\left\| \Phi(t_n, \cdot) \right\|_{l^2}},  \label{relative-error-fun}  \\
    & e_{\rho}^r (t_n) = \frac{\left\| \rho(t_n, \cdot) - \rho_I^n(\cdot) \right\|_{l^2}}{\left\| \rho(t_n, \cdot) \right\|_{l^2}}, \label{relative-error-position} \\
    & e_{\textbf{J}}^r (t_n) = \frac{\left\| \textbf{J} (t_n, \cdot) - \textbf{J}_I^n(\cdot) \right\|_{l^2}}{\left\| \textbf{J}(t_n, \cdot) \right\|_{l^2}}, \label{relative-error-curr}
\end{align}
\noindent
where the total probability density function $\rho$ and the current density function $\textbf{J}$ are defined in \eqref{def_total_probability_density} and \eqref{def_ current density}.

\subsection{An example in 1D}
\noindent
First, we consider a one-dimensional case with $d = 1$. The electromagnetic potentials are given by $V(x) = \frac{1 - x}{1 + x^2}, A_1(x) = \frac{(x + 1)^2}{1 + x^2}$ and the initial condition is chosen as follows:

\begin{align}
    & \phi_1(0,x) = \frac{1}{2} e^{-4x^2 + i \frac{S_0(x)}{\epsilon}} \left( 1 + \sqrt{1 + S_0'(x)^2} \right), \\
    & \phi_2(0,x) = \frac{1}{2} e^{-4x^2 + i \frac{S_0(x)}{\epsilon}}  S_0'(x),
\end{align}
\noindent
where $S_0(x) = \frac{1}{40} \left(1 + \cos(2 \pi x) \right)$ and $\epsilon \in (0,1]$.

\noindent
In our numerical experiments, we take the spatial domain $\Omega = (-16,16)$ and we compute the solutions until $t = 2$. Moreover, as there is no explicit expression for the exact solution, we use the numerical solution obtained via the $S_\text{4c}$ method with time step size $\tau = \frac{0.1}{4^{8}}$ and spatial grid size $h = \frac{1}{2^{10}}$ as the ``exact'' solution for comparison.

\noindent
Firstly, we test the dependence of the error on time step sizes and present the relative errors of $S_1$ and $S_2$ in Tables \ref{tab:error_fun_tau_$S_1$} and \ref{tab:error_fun_tau_$S_2$}. Tables \ref{tab:error_fun_tau_$S_1$} and \ref{tab:error_fun_tau_$S_2$} clearly show the first order convergence in time for $S_1$ and the second order convergence in time for $S_2$, when the spatial error is negligible. Moreover, the elements of the diagonal lines of the two tables exhibit that the $\epsilon$-scalability of $S_1$ and $S_2$ are $\tau = O(\epsilon^2)$ and $\tau = O(\epsilon^{\frac{3}{2}})$, respectively.

\begin{table}[htbp]
\centering
\caption{Discrete $l^2$ relative temporal errors $e^r_{\Phi}(t = 2)$ of $S_1$ under different $\tau$ and $\epsilon$ for the Dirac equation \eqref{def_Dirac_simplified} in 1D.}
\label{tab:error_fun_tau_$S_1$}
\begin{tabular}{l *{6}{S[table-format=1.2e-1]}}
\toprule
\multirow{2}{*}{$e^r_{\Phi}(t = 2)$} 
& \multicolumn{1}{l}{\makecell{$\tau_0=0.1$ \\ $h_0=1/2^{10}$ }}
& \multicolumn{1}{l}{\makecell{$\tau_0/4$} }
& \multicolumn{1}{l}{\makecell{$\tau_0/4^2$ } } 
& \multicolumn{1}{l}{\makecell{$\tau_0/4^3$ } }
& \multicolumn{1}{l}{\makecell{$\tau_0/4^4$} } 
& \multicolumn{1}{l}{\makecell{$\tau_0/4^5$} }\\
 \cmidrule(lr){2-2} \cmidrule(lr){3-3} \cmidrule(lr){4-4} \cmidrule(lr){5-5} \cmidrule(l){6-6} \cmidrule(l){7-7}
$\epsilon_0=1$       &0.096048069366561	&0.023713304251863	&0.005909816661789	&0.001476302045016	&0.000369003588005	&0.000092246405849
	\\
    {order}             &\multicolumn{1}{c}{-} &1.009029955305683	&1.002255712206628	&1.000562727049326	&1.000140585630060	&1.000035119010315
\\
\midrule
$\epsilon_0/2$            &0.217568260980917	&0.050026637270193	&0.012247290620450	&0.003045926323150	&0.000760493292795	&0.000190061638516

\\
{order}           &\multicolumn{1}{c}{-} &1.060349863011445	&1.015116925924287	&1.003754838509917	&1.000936808097813	&1.000234076030059
\\

\midrule
$\epsilon_0/2^2$            &0.486073740730603	&0.088994967686216	&0.020360380839638	&0.004978618448642	&0.001237802611769	&0.000309024253552

\\
{order}             &\multicolumn{1}{c}{-} &1.224689766201960	&1.063979606247818	&1.015973593847084	&1.003982091612305	&1.000994647278682
\\

\midrule
$\epsilon_0/2^3$              &0.943664921734977	&0.199144168507637	&0.035191660391577	&0.007998839048507	&0.001952667178534	&0.000485277226218

 \\
{order}              &\multicolumn{1}{c}{-}   &1.122230720753122	&1.250253872524423	&1.068685528253053	&1.017172275878272	&1.004282504364117
\\
\midrule
$\epsilon_0/2^4$            &1.451949246876376	&0.710246167592942	&0.095454720923218	&0.016603481560620	&0.003745531740630	&0.000912274966021

 \\
{order}              &\multicolumn{1}{c}{-}   &\multicolumn{1}{l}{0.52}	&1.447715342827507	&1.261665380534816	&1.074121668524360	&1.018814957225517
\\
\midrule
$\epsilon_0/2^5$            &1.334557081683848	&1.192495431268675	&0.412763172973088	&0.046259010383953	&0.007974621295882	&0.001792368602604

 \\
{order}              &\multicolumn{1}{c}{-}   &\multicolumn{1}{l}{0.08}	&\multicolumn{1}{l}{0.77}	&1.578753975140230	&1.268123244432980	&1.076774324545238
\\

\bottomrule
\end{tabular}

\end{table}

\begin{table}[htbp]
\centering
\caption{Discrete $l^2$ relative temporal errors $e^r_{\Phi}(t = 2)$ of $S_2$ under different $\tau$ and $\epsilon$ for the Dirac equation \eqref{def_Dirac_simplified} in 1D.}
\label{tab:error_fun_tau_$S_2$}
\begin{tabular}{l *{6}{S[table-format=1.2e-1]}}
\toprule
\multirow{2}{*}{$e^r_{\Phi}(t = 2)$} 
& \multicolumn{1}{l}{\makecell{$\tau_0=0.1/4$ \\ $h_0=1/2^{10}$ }}
& \multicolumn{1}{l}{\makecell{$\tau_0/4$} }
& \multicolumn{1}{l}{\makecell{$\tau_0/4^2$ } } 
& \multicolumn{1}{l}{\makecell{$\tau_0/4^3$ } }
& \multicolumn{1}{l}{\makecell{$\tau_0/4^4$} } 
& \multicolumn{1}{l}{\makecell{$\tau_0/4^5$} }\\
\cmidrule(lr){2-2} \cmidrule(lr){3-3} \cmidrule(lr){4-4} \cmidrule(lr){5-5} \cmidrule(l){6-6} \cmidrule(l){7-7}

$\epsilon_0=1$       &3.144415512967116e-04	&1.965192811441731e-05	&1.228252768691590e-06	&7.677857729757565e-08	&4.856977275575182e-09	&7.391775746406983e-10

 \\
{order}             &\multicolumn{1}{c}{-} &2.000024550213926	&1.999995735316167	&1.999879926784160	&1.991286592926328	&1.358032921368678
\\

\midrule
$\epsilon_0/4^{2/3}$            &0.004046882538900	&0.000253073317691	&0.000015817636717	&0.000000988603192	&0.000000061790070	&0.000000003919100

\\
{order}           &\multicolumn{1}{c}{-} &1.999591831266073	&1.999974718466250	&1.999999345200737	&1.999972329251264	&1.989390263848456

\\

\midrule
$\epsilon_0/4^{4/3}$            &0.048066940518475	&0.003010328451840	&0.000188160162771	&0.000011760059583	&0.000000734998787	&0.000000045937033

\\
{order}             &\multicolumn{1}{c}{-} &1.998526083768575	&1.999943893580050	&1.999996969232388	&2.000004845466152	&2.000006139357000

\\

\midrule
$\epsilon_0/4^2$              &0.571677017113967	&0.044772153010597	&0.002802749880465	&0.000175180778303	&0.000010948804034	&0.000000684273971

 \\
{order}              &\multicolumn{1}{c}{-}   &1.837263344905292	&1.998843394220627	&1.999963306887665	&1.999999644850004	&2.000027704119271

\\

\midrule
$\epsilon_0/4^{8/3}$            &1.449283435969003	&0.564855011992282	&0.044097987199231	&0.002759337307017	&0.000172462626057	&0.000010778911447

 \\
{order}              &\multicolumn{1}{c}{-} &\multicolumn{1}{l}{0.68}	&1.839547944594967	&1.999159538300181	&1.999983083702146	&2.000000179422990

\\

\midrule
$\epsilon_0/4^{10/3}$            &1.418049967289920	&1.407600186375704	&0.565104180664426	&0.044009937229629	&0.002753313356493	&0.000172084133693

 \\
{order}              &\multicolumn{1}{c}{-}   &\multicolumn{1}{l}{0.01}	&\multicolumn{1}{l}{0.66}	&1.841307821591548	&1.999294297771113	&1.999991411258088
\\

\bottomrule
\end{tabular}
\end{table}

\begin{table}[htbp] 
\centering
\caption{Discrete $l^2$ relative temporal errors $e^r_{\Phi}(t = 2)$ of $S_\text{4c}$ under different $\tau$ and $\epsilon$ for the Dirac equation \eqref{def_Dirac_simplified} in 1D.}
\label{tab:error_fun_tau_S4c}
\begin{tabular}{l *{6}{S[table-format=1.2e-1]}}
\toprule
\multirow{2}{*}{$e^r_{\Phi}(t = 2)$} 
& \multicolumn{1}{l}{\makecell{$\tau_0=0.1/4$  }}
& \multicolumn{1}{l}{\makecell{$\tau_0/4$} }
& \multicolumn{1}{l}{\makecell{$\tau_0/4^2$ } } 
& \multicolumn{1}{l}{\makecell{$\tau_0/4^3$ } }
& \multicolumn{1}{l}{\makecell{$\tau_0/4^4$} } 
& \multicolumn{1}{l}{\makecell{$\tau_0/4^5$} }\\
\cmidrule(lr){2-2} \cmidrule(lr){3-3} \cmidrule(lr){4-4} \cmidrule(lr){5-5} \cmidrule(l){6-6} \cmidrule(l){7-7}

$\epsilon_0 = 1/4^{4/5}$            &3.588985314539769e-06	&1.403036727949944e-08	&6.927227104318929e-10	&6.883146949190195e-10	&6.800459933980934e-10	&6.499003099701679e-10

\\
{order}             &\multicolumn{1}{c}{-} &3.999439717072379	&2.170065496484900	&\multicolumn{1}{l}{0}	&\multicolumn{1}{l}{0.01}	&\multicolumn{1}{l}{0.03}

\\

\midrule

$\epsilon_0/4^{4/5}$            &4.331690752119641e-04	&1.682821435617751e-06	&6.637855560267670e-09	&6.810401792777856e-10	&6.713204397919403e-10	&6.393265903746217e-10

\\
{order}             &\multicolumn{1}{c}{-} &4.003952169663850	&3.992974574097295	&1.642452708139875	&\multicolumn{1}{l}{0.01}	&\multicolumn{1}{l}{0.04}

\\

\midrule
$\epsilon_0/4^{8/5}$              &0.095837738076884	&0.000353447611445	&0.000001376074617	&0.000000005408022	&0.000000000650725	&0.000000000618273

 \\
{order}              &\multicolumn{1}{c}{-}   &4.041476828013084	&4.002396940495247	&3.995620946895298	&1.527490361592509	&\multicolumn{1}{l}{0.04}

\\

\midrule
$\epsilon_0/4^{12/5}$            &1.291441223223771	&0.091947897968323	&0.000349173697547	&0.000001361360196	&0.000000005348780	&0.000000000584974

 \\
{order}              &\multicolumn{1}{c}{-}   &1.906010790304366	&4.020363947507009	&4.001376122368487	&3.995811663093173	&1.596383136731600

\\

\midrule
$\epsilon_0/4^{16/5}$            &1.419932978466424	&1.181880734239523	&0.090457013213404	&0.000348787405875	&0.000001360973920	&0.000000005344422

 \\
{order}              &\multicolumn{1}{c}{-}   &\multicolumn{1}{l}{0.13}	&1.853854143499580	&4.009370301313804	&4.000782358400063	&3.996194938151118

\\

\midrule
$\epsilon_0/4^4$            &1.418658871683228	&1.415267316774206	&1.116143831199114	&0.089675907911070	&0.000348607216425	&0.000001360876852

 \\

{order}              &\multicolumn{1}{c}{-}   &\multicolumn{1}{l}{0}	&\multicolumn{1}{l}{0.17}	&1.818829347162813	&4.003487100886877	&4.000461051559708
\\

\midrule
$\epsilon_0/4^{24/5}$            &1.413986777607456	&1.410929417297270	&1.416350656777041	&1.061846705500867	&0.031205512574174	&0.000124465824532

 \\

{order}           &\multicolumn{1}{c}{-} &\multicolumn{1}{l}{0}	&\multicolumn{1}{l}{0}	&\multicolumn{1}{l}{0.21}	&2.544315392838269	&3.984953716704456
\\

\bottomrule
\end{tabular}
\end{table}

\noindent
In addition, the relative error of the wave function obtained via the $S_\text{4c}$ method in Table \eqref{tab:error_fun_tau_S4c} demonstrates the fourth-order accuracy in time. Compared to $S_1$ and $S_2$, $S_\text{4c}$ achieves higher accuracy with less stringent time step size requirements owing to the higher $\epsilon$-scalability.

\noindent
Furthermore, we compute the relative temporal errors of the total probability density $e^r_{\rho}$ and the current density $e^r_{\textbf{J}}$ for $S_1$ and $S_2$. For simplicity, we omit the results of $S_2$ here. The data of Table \eqref{tab:error_position_tau_$S_1$} and \eqref{tab:error_curr_tau_$S_1$} confirms the error bounds in Corollary \eqref{corollary-sp1}.

\begin{table}[htbp]
\vspace{10pt}
\centering
\caption{Discrete $l^2$ relative temporal errors $e^r_{\rho}(t = 2)$ of $S_{1}$ under different $\tau$ and $\epsilon$ for the Dirac equation \eqref{def_Dirac_simplified} in 1D.}
\label{tab:error_position_tau_$S_1$}
\begin{tabular}{l *{6}{S[table-format=1.2e-1]}}
\toprule
\multirow{2}{*}{$e^r_{\rho}(t = 2)$} 
& \multicolumn{1}{l}{\makecell{$\tau_0=0.1$ }}
& \multicolumn{1}{l}{\makecell{$\tau_0/4$} }
& \multicolumn{1}{l}{\makecell{$\tau_0/4^2$ } } 
& \multicolumn{1}{l}{\makecell{$\tau_0/4^3$ } }
& \multicolumn{1}{l}{\makecell{$\tau_0/4^4$} } 
& \multicolumn{1}{l}{\makecell{$\tau_0/4^5$} }\\
 \cmidrule(lr){2-2} \cmidrule(lr){3-3} \cmidrule(lr){4-4} \cmidrule(lr){5-5} \cmidrule(l){6-6} \cmidrule(l){7-7}
$\epsilon_0=1$       &0.045539890889608	&0.010750171335365	&0.002650927057497	&0.000660493988877	&0.000164984679931	&0.000041237750414
\\

{order}             &\multicolumn{1}{c}{-} &1.041385590663461	&1.009895387333536	&1.002439818514095	&1.000606683016035	&1.000147263498727
\\
\midrule
$\epsilon_0/2$            &0.072019811609535	&0.014122894525347	&0.003309635993832	&0.000814094815194	&0.000202700213596	&0.000050623831667

\\

{order}           &\multicolumn{1}{c}{-} &1.175179011277318	&1.046645672646034	&1.011701908487546	&1.002924610806386	&1.000729498103903
\\

\midrule
$\epsilon_0/2^2$            &0.169458316496618	&0.024604419137675	&0.005228702416947	&0.001253767888764	&0.000310175529043	&0.000077341240449

\\

{order}             &\multicolumn{1}{c}{-} &1.391970539428897	&1.117196294325928	&1.030091339264861	&1.007556754610062	&1.001887531211056
\\

\midrule
$\epsilon_0/2^3$              &0.362536086103019	&0.046795564567538	&0.008255586761282	&0.001873074091722	&0.000456956158888	&0.000113543144108

 \\

{order}              &\multicolumn{1}{c}{-}   &1.476840453511141	&1.251464565206472	&1.069981394636230	&1.017640152861510	&1.004407580337979
\\

\midrule
$\epsilon_0/2^4$            &1.197078846813399	&0.126599431984464	&0.014362082414649	&0.002536163019603	&0.000572608294333	&0.000139425795690

 \\

{order}              &\multicolumn{1}{c}{-}   &1.620587671526008	&1.569967040184086	&1.250772779575126	&1.073513504389930	&1.019025531925884
\\

\midrule
$\epsilon_0/2^5$            &0.985634411277876	&0.704249856473499	&0.061131584006285	&0.006943465891570	&0.001231380128008	&0.000277253187628

 \\

{order}              &\multicolumn{1}{c}{-}   &\multicolumn{1}{l}{0.24}	&1.763048754799060	&1.569095034173053	&1.247689892371584	&1.075500119577001
\\

\bottomrule
\end{tabular}
\vspace{10pt}
\end{table}

\begin{table}[htbp]
\vspace{10pt}
\centering
\caption{Discrete $l^2$ relative temporal errors $e^r_{J}(t = 2)$ of $S_{1}$ under different $\tau$ and $\epsilon$ for the Dirac equation \eqref{def_Dirac_simplified} in 1D.}
\label{tab:error_curr_tau_$S_1$}
\begin{tabular}{l *{6}{S[table-format=1.2e-1]}}
\toprule
\multirow{2}{*}{$e^r_{\textbf{J}}(t = 2)$} 
& \multicolumn{1}{l}{\makecell{$\tau_0=0.1$ \\ $h_0=1/2^{10}$ }}
& \multicolumn{1}{l}{\makecell{$\tau_0/4$} }
& \multicolumn{1}{l}{\makecell{$\tau_0/4^2$ } } 
& \multicolumn{1}{l}{\makecell{$\tau_0/4^3$ } }
& \multicolumn{1}{l}{\makecell{$\tau_0/4^4$} } 
& \multicolumn{1}{l}{\makecell{$\tau_0/4^5$} }\\
 \cmidrule(lr){2-2} \cmidrule(lr){3-3} \cmidrule(lr){4-4} \cmidrule(lr){5-5} \cmidrule(l){6-6} \cmidrule(l){7-7}
$\epsilon_0=1$       &0.066753700387697	&0.015814120722082	&0.003901747405824	&0.000972251475968	&0.000242864856530	&0.000060704108405

	\\
{order}             &\multicolumn{1}{c}{-} &1.038817234751467	&1.009510527618141	&1.002359478902319	&1.000587889098714	&1.000143838139084
\\
\midrule
$\epsilon_0/2$            &0.115086287577867	&0.022686826517850	&0.005329207220075	&0.001311791289422	&0.000326681216677	&0.000081591457225

\\
{order}           &\multicolumn{1}{c}{-} &1.171394613507809	&1.044930989500373	&1.011191365664406	&1.002791394670613	&1.000696746765155
\\

\midrule
$\epsilon_0/2^2$            &0.311501102169513	&0.051269536203625	&0.011251622092889	&0.002719437459598	&0.000674103112334	&0.000168167635335

\\
{order}             &\multicolumn{1}{c}{-} &1.301531757965911	&1.093984421954364	&1.024376425039996	&1.006133528221281	&1.001534603711182
\\

\midrule
$\epsilon_0/2^3$              &0.465974003799915	&0.095957027541591	&0.017185306976403	&0.003899002554018	&0.000950692544403	&0.000236183767265

 \\
{order}              &\multicolumn{1}{c}{-}   &1.139894548787807	&1.240606422795721	&1.069999308763582	&1.018027174092598	&1.004534516840981
\\

\midrule
$\epsilon_0/2^4$            &1.287764692939789	&0.161738255937554	&0.023748912985686	&0.004004267072325	&0.000888323820606	&0.000215266228070

 \\
{order}              &\multicolumn{1}{c}{-}   &1.496568067979061	&1.383863786711558	&1.284125687680823	&1.086190308825080	&1.022481838577901
\\

\midrule
$\epsilon_0/2^5$            &0.992736950866205	&0.694593095224064	&0.085293598117317	&0.010105458070550	&0.001852399215342	&0.000422152123014

 \\
{order}              &\multicolumn{1}{c}{-}   &\multicolumn{1}{l}{0.26}	&1.512829351490850	&1.538651371034586	&1.223833881501067	&1.066780088225304
\\

\bottomrule
\end{tabular}
\vspace{10pt}
\end{table}

\noindent
And to make the dependence of various errors on time step sizes more explicit, we present the following convergence analysis plot \eqref{fig-$S_2$-ep1} and \eqref{fig-$S_2$-ep16} with the dimensionless parameter $\epsilon = 1$, $ \frac{1}{16}$ respectively. 

\begin{figure}[htbp]
    \centering

    \subfloat[$\epsilon = 1$]{
        \includegraphics[width=0.48\textwidth]{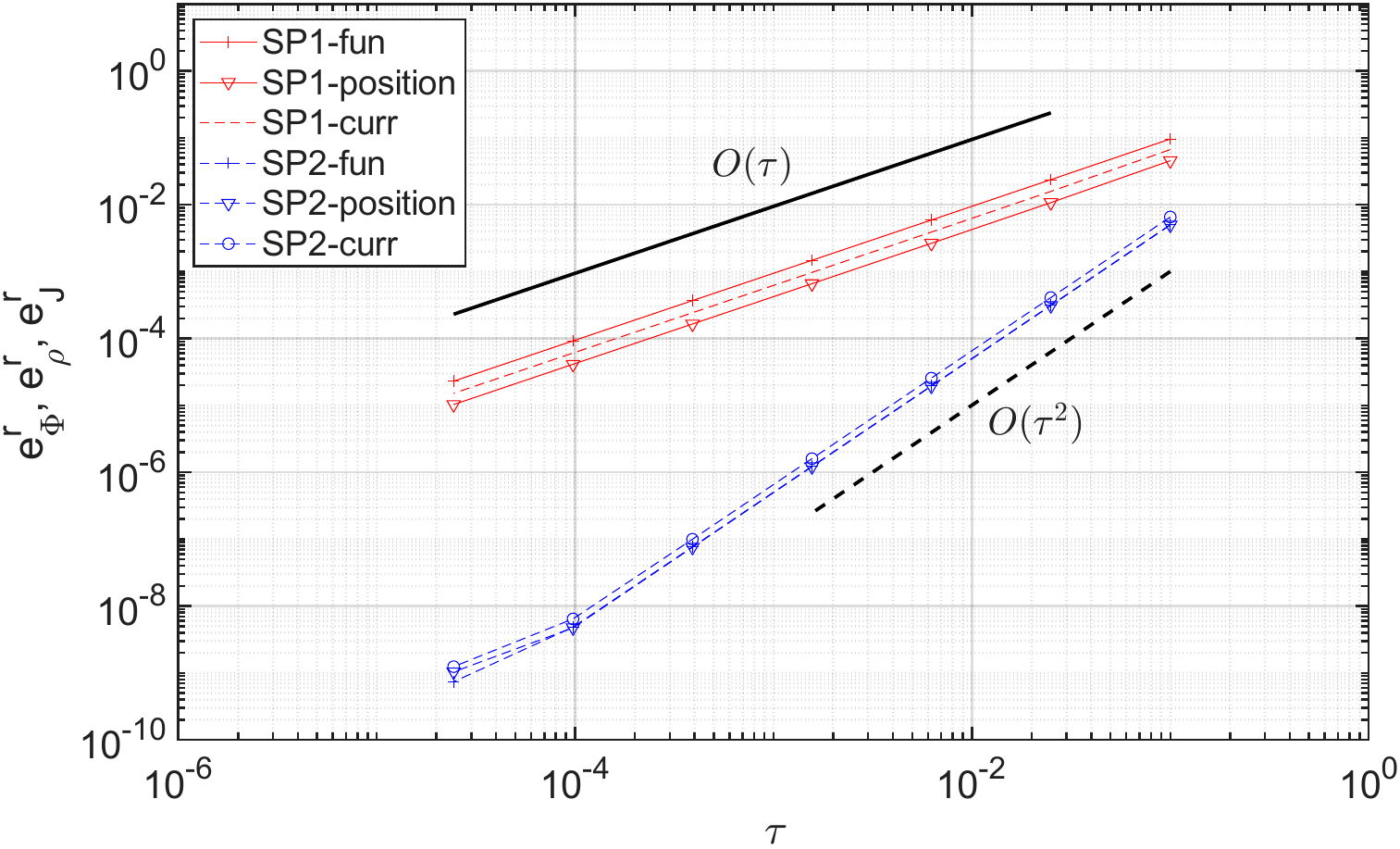}
        \label{fig-$S_2$-ep1}
    }
    \hfill
    \subfloat[$\epsilon = \frac{1}{16}$]{
        \includegraphics[width=0.48\textwidth]{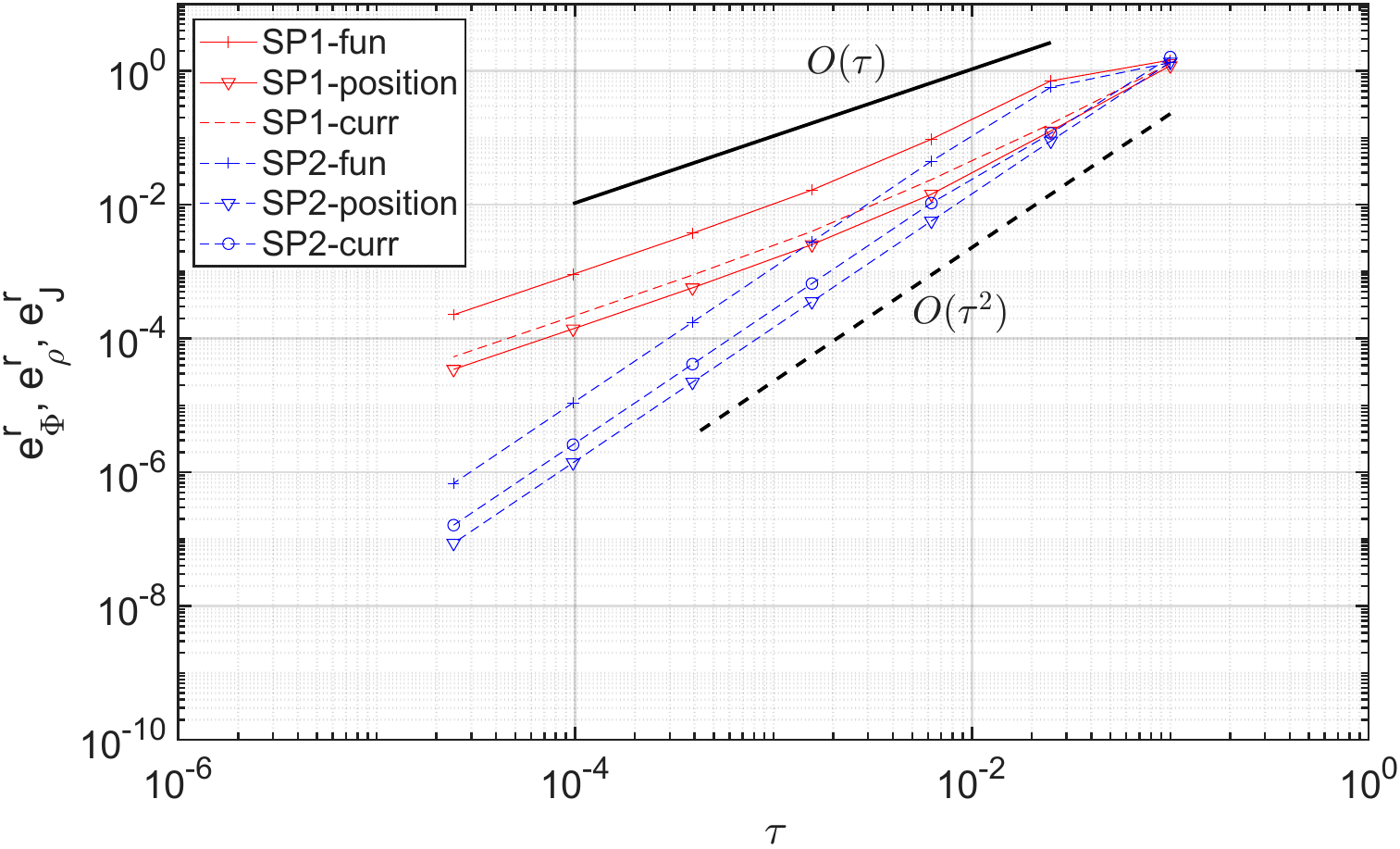}
        \label{fig-$S_2$-ep16}
    }

    \caption{Error comparison of $S_1$ and $S_2$ methods for the Dirac equation \eqref{def_Dirac_simplified} in 1D under different time step sizes.}
\end{figure}

\noindent
Figures~\ref{fig-$S_2$-ep1} and~\ref{fig-$S_2$-ep16}
display the relationships between the relative errors
\(
e_{\Phi}^r,\ e_{\rho}^r,\ \text{and}\ e_{J}^r
\)
and the time step size \(\tau\)
for \(\epsilon = 1\) and
\(\epsilon = \frac{1}{16}\), respectively.
It can be observed from these figures that
the relative errors
\(
e_{\Phi}^r,\ e_{\rho}^r,\ \text{and}\ e_{J}^r
\)
exhibit the same convergence order for different values of
\(\epsilon\),
indicating a consistent numerical behavior of the scheme
with respect to different physical quantities.

\noindent
Furthermore, the \(S_2\) method achieves a higher convergence rate
than the \(S_1\) method.
To attain an accuracy of order \(10^{-6}\),
smaller values of \(\epsilon\) require finer time step sizes \(\tau\).
In addition, the \(S_1\) method demands significantly smaller
time step sizes than the \(S_2\) method
to achieve the same level of accuracy,
demonstrating the superior efficiency and accuracy balance
of the \(S_2\) scheme.

\noindent
Finally, Table~\ref{tab:error_fun_h_$S_1$}
reports the spatial convergence behavior of the \(S_1\) method
with the time step size fixed at $\tau = \frac{0.1}{4^{8}}$.
It can be observed that, as the spatial mesh size \(h\) decreases,
the relative error exhibits a typical spectral convergence behavior.
This result confirms the high spatial accuracy of the proposed scheme.

\begin{table}[htbp]
\vspace{10pt}
\centering
\caption{Discrete $l^2$ relative spatial errors $e^r_{\Phi}(t = 2)$ of $S_{1}$ under different $h$ and $\epsilon$ for the Dirac equation \eqref{def_Dirac_simplified} in 1D.}
\label{tab:error_fun_h_$S_1$}
\begin{tabular}{l *{5}{S[table-format=1.2e-1]}}
\toprule
\multirow{2}{*}{$e^r_{\Phi}(t = 2)$} 
& \multicolumn{1}{l}{\makecell{$\tau_0= \frac{0.1}{4^{8}}$ \\ $h_0=1$ }}
& \multicolumn{1}{l}{\makecell{$h_0/2$} }
& \multicolumn{1}{l}{\makecell{$h_0/2^2$ } } 
& \multicolumn{1}{l}{\makecell{$h_0/2^3$ } }
& \multicolumn{1}{l}{\makecell{$h_0/2^4$} } \\
\cmidrule(lr){2-2} \cmidrule(lr){3-3} \cmidrule(lr){4-4} \cmidrule(lr){5-5} \cmidrule(l){6-6}

$\epsilon_0 = 1$            &0.929953054769695	&0.201962882505403	&0.009521839888527	&0.000006810490029	&0.000001441339675

 \\
 \midrule
 $\epsilon_0/2$            &0.996396303469822	&0.749107886574744	&0.053145655078847	&0.000088751420467	&0.000002969399845

 \\
 \midrule
 $\epsilon_0/2^2$            &1.282613750928035	&0.996910027792905	&0.512041745174147	&0.003810831166543	&0.000004826340150

 \\
 \midrule
$\epsilon_0/2^3$            &1.664836697610778	&1.255981451057587	&0.728854631122200	&0.275583248851465	&0.000020515007198

 \\
\bottomrule
\end{tabular}
\vspace{10pt}
\end{table}

\subsection{An example in 2D}
\noindent
We now consider an example in 2D with the electromagnetic potentials:
\begin{align}
    & V(\textbf{x}) = \cos \left( \frac{4 \pi}{\sqrt{3}} x_1 \right) + \cos \left[ \frac{4 \pi}{\sqrt{3}} \left( \frac{1}{2} x_1 + \frac{\sqrt{3}}{2} x_2 \right) \right] + \cos \left[ \frac{4 \pi}{\sqrt{3}} \left( \frac{1}{2} x_1 - \frac{\sqrt{3}}{2} x_2 \right) \right], \\
    & A_1(\textbf{x}) = 0 \quad, \qquad A_2(\textbf{x}) = x_1,
\end{align}
where $\textbf{x} = (x_1, x_2) \in \Omega = (-10,10) \times (-10,10)$. And the initial condition is given as follows:
\begin{equation}
\phi_1(0,\textbf{x}) = e^{-\frac{x_1^2 + x_2^2}{2}}, \quad \phi_2(0,\textbf{x}) = e^{-\frac{\left(x_1 - 1 \right)^2 + x_2^2}{2}}.
\end{equation}
 
Similar to the example in 1D, we use the numerical solution obtained via the $S_\text{4c}$ method with $\tau = \frac{0.1}{4^{5}}$ and $h = \frac{1}{32}$ as the ``exact'' solution. And we employ the relative errors introduced in \eqref{relative-error-fun} to evaluate the results.

\noindent
To exhibit the temporal error, we always take the mesh size $ h = \frac{1}{32}$ in the numerical experiments such that the spatial error could be neglected. The Table \ref{tab:error_fun_tau_$S_1$_2D} and \ref{tab:error_fun_tau_$S_2$_2D} present the temporal errors of $S_1$ and $S_2$. We can clearly identify that $S_1$ is first-order accurate in time, while $S_2$ achieves second-order accuracy. Moreover, the numerical results in the Table \eqref{tab:error_fun_tau_S4c_2D} demonstrate that the fourth-order accuracy is preserved in 2D, despite the introduction of additional errors arising from the implementation of the RK4 method.

\begin{table}[htbp]
\centering
\caption{Discrete $l^2$ relative temporal errors $e^r_{\Phi}(t = 0.2)$ of $S_{1}$ under different $h$ and $\epsilon$ for the Dirac equation \eqref{def_Dirac_simplified} in 2D.}
\label{tab:error_fun_tau_$S_1$_2D}
\begin{tabular}{l *{6}{S[table-format=1.2e-1]}}
\toprule
\multirow{2}{*}{$e^r_{\Phi}(t = 0.2)$} 
& \multicolumn{1}{l}{\makecell{$\tau_0=0.1$ \\ $h_x=h_y=1/2^{5}$ }}
& \multicolumn{1}{l}{\makecell{$\tau_0/4$} }
& \multicolumn{1}{l}{\makecell{$\tau_0/4^2$ } } 
& \multicolumn{1}{l}{\makecell{$\tau_0/4^3$ } }
& \multicolumn{1}{l}{\makecell{$\tau_0/4^4$} } 
& \multicolumn{1}{l}{\makecell{$\tau_0/4^5$} }\\
 \cmidrule(lr){2-2} \cmidrule(lr){3-3} \cmidrule(lr){4-4} \cmidrule(lr){5-5} \cmidrule(l){6-6} \cmidrule(l){7-7}
$\epsilon_0=1$       &0.084903029839769	&0.021091072222317	&0.005271458333811	&0.001317893766976	&0.000329477008277	&0.000082369502242	\\

{order}             &\multicolumn{1}{c}{-} &1.004591799101960	&1.000179200354726	&0.999984026257693	&0.999992191493750	&0.999997809116941
\\
\midrule
$\epsilon_0/2$            &0.176449171731901	&0.043825821994401	&0.010956319292118	&0.002739361781205	&0.000684862675445	&0.000171217130542

\\
{order}           &\multicolumn{1}{c}{-} &1.004699801272384	&1.000008967573395	&0.999925749019240	&0.999976585204735	&0.999993841831243
\\

\midrule
$\epsilon_0/2^2$            &0.367385142891962	&0.089945961290401	&0.022443860852467	&0.005609579691177	&0.001402328636173	&0.000350578334930

\\
{order}             &\multicolumn{1}{c}{-} &1.015081437988956	&1.001368815110557	&1.000178145148868	&1.000034096543619	&1.000007868426092
\\

\midrule
$\epsilon_0/2^3$              &0.688878071493023	&0.155219778785632	&0.037965934041951	&0.009448356071121	&0.002359540340532	&0.000589728098429
 \\
{order}              &\multicolumn{1}{c}{-}   &1.074968125413198	&1.015767498359072	&1.003285130804329	&1.000778749051569	&1.000191998482274
\\
\midrule
$\epsilon_0/2^4$            &1.322128205491601	&0.358138442479074	&0.078612323766420	&0.019133935985668	&0.004753844729211	&0.001186653599460
 \\
{order}              &\multicolumn{1}{c}{-}   &\multicolumn{1}{l}{0.94}	&1.093844992048895	&1.019310909736713	&1.004483494561442	&1.001097964328407
\\

\bottomrule
\end{tabular}

\end{table}

\begin{table}[htbp]
\centering
\caption{Discrete $l^2$ relative temporal errors $e^r_{\Phi}(t = 0.2)$ of $S_{2}$ under different $h$ and $\epsilon$ for the Dirac equation \eqref{def_Dirac_simplified} in 2D.}
\label{tab:error_fun_tau_$S_2$_2D}
\begin{tabular}{l *{6}{S[table-format=1.2e-1]}}
\toprule
\multirow{2}{*}{$e^r_{\Phi}(t = 0.2)$} 
& \multicolumn{1}{l}{\makecell{$\tau_0=0.1$ \\ $h_x=h_y=1/2^{5}$ }}
& \multicolumn{1}{l}{\makecell{$\tau_0/4$} }
& \multicolumn{1}{l}{\makecell{$\tau_0/4^2$ } } 
& \multicolumn{1}{l}{\makecell{$\tau_0/4^3$ } }
& \multicolumn{1}{l}{\makecell{$\tau_0/4^4$} } 
& \multicolumn{1}{l}{\makecell{$\tau_0/4^5$} }\\
 \cmidrule(lr){2-2} \cmidrule(lr){3-3} \cmidrule(lr){4-4} \cmidrule(lr){5-5} \cmidrule(l){6-6} \cmidrule(l){7-7}
$\epsilon_0=1$       &0.010588539113361	&0.000654255410691	&0.000040862240422	&0.000002553778002	&0.000000159610699	&0.000000009975673	\\
{order}             &\multicolumn{1}{c}{-} &2.008252897857259	&2.000506869010266	&2.000031642027781	&2.000001925217947	&1.999999663860311
\\
\midrule
$\epsilon_0/4^{\frac{2}{3}} $            &0.034174625402725	&0.002081378833625	&0.000129883274271	&0.000008116914481	&0.000000507304075	&0.000000031706487

\\
{order}           &\multicolumn{1}{c}{-} &2.018657021923595	&2.001126005247817	&2.000070217969804	&2.000004379343485	&2.000000406203295
\\

\midrule
$\epsilon_0/4^{\frac{4}{3}} $            &0.193461028878049	&0.010407916799986	&0.000643970933758	&0.000040223042650	&0.000002513842053	&0.000000157114748

\\
{order}             &\multicolumn{1}{c}{-} &2.108144868563074	&2.007270975906573	&2.000450724818045	&2.000028153001086	&2.000001746682796
\\

\midrule
$\epsilon_0/4^2$              &1.141213786476053	&0.120984356260193	&0.007117188744656	&0.000443111376171	&0.000027687778890	&0.000001730459288
 \\
{order}              &\multicolumn{1}{c}{-}   &1.618838330990825	&2.043684603624311	&2.002783112409317	&2.000174067759654	&2.000011210247900
\\
\midrule
$\epsilon_0/4^{\frac{8}{3}} $          &1.175858263141160	&1.082176671497658	&0.106177732522864	&0.006496749308825	&0.000405408585967	&0.000025313786564
 \\
{order}              &\multicolumn{1}{c}{-}   &\multicolumn{1}{l}{0.06}	&1.674691452001695	&2.015309695242468	&2.001134744153255	&2.000690704510041
\\

\bottomrule
\end{tabular}

\end{table}

\begin{table}[htbp]
\centering
\caption{Discrete $l^2$ relative temporal errors $e^r_{\Phi}(t = 0.2)$ of $S_\text{4c}$ under different $h$ and $\epsilon$ for the Dirac equation \eqref{def_Dirac_simplified} in 2D ($e_{\Phi}^r(t=0.2)$)}
\label{tab:error_fun_tau_S4c_2D}
\begin{tabular}{l *{6}{S[table-format=1.2e-1]}}
\toprule
\multirow{2}{*}{$e^r_{\Phi}(t = 0.2)$} 
& \multicolumn{1}{l}{\makecell{$\tau_0=0.1$ \\ $h_x=h_y=1/2^{5}$ }}
& \multicolumn{1}{l}{\makecell{$\tau_0/2$} }
& \multicolumn{1}{l}{\makecell{$\tau_0/2^2$ } } 
& \multicolumn{1}{l}{\makecell{$\tau_0/2^3$ } }
& \multicolumn{1}{l}{\makecell{$\tau_0/2^4$} } 
& \multicolumn{1}{l}{\makecell{$\tau_0/2^5$} }\\
 \cmidrule(lr){2-2} \cmidrule(lr){3-3} \cmidrule(lr){4-4} \cmidrule(lr){5-5} \cmidrule(l){6-6} \cmidrule(l){7-7}
$\epsilon_0=1$       &4.880905294362108e-05	&2.883457180125475e-06	&1.777528472754915e-07	&1.106915018924691e-08	&6.911031922499417e-10	&4.318945030418775e-11	\\
{order}             &\multicolumn{1}{c}{-} &4.081277255590376	&4.019855023786688	&4.005256297763113	&4.001499513187824	&4.000150283273003
\\
\midrule
$\epsilon_0/2^{\frac{4}{5}} $            &1.709688548204713e-04	&9.607416873433098e-06	&5.836444311342135e-07	&3.615305805758121e-08	&2.252333713280032e-09	&1.405718711852693e-10

\\
{order}           &\multicolumn{1}{c}{-} &4.153441136230036	&4.040986967984627	&4.012900125715115	&4.004625181859796	&4.002040756575982
\\

\midrule
$\epsilon_0/2^{\frac{8}{5}} $            &0.001197718285284	&0.000067301402426	&0.000004103975752	&0.000000254569466	&0.000000015867545	&0.000000000990621

\\
{order}             &\multicolumn{1}{c}{-} &4.153508234214018	&4.035542456142731	&4.010890915788260	&4.003908573409788	&4.001601342408885
\\

\midrule
$\epsilon_0/2^{\frac{12}{5}}$              &0.015155365245978	&0.000881926912620	&0.000054634326220	&0.000003406643911	&0.000000212749255	&0.000000013292809
 \\
{order}              &\multicolumn{1}{c}{-}   &4.103025710257935	&4.012779528663052	&4.003384608664907	&4.001125165566361	&4.000436140822420
\\

\midrule
$\epsilon_0/2^{\frac{16}{5}} $          &0.224005331960773	&0.013170289295190	&0.000865349370416	&0.000054396679535	&0.000003403700733	&0.000000212806691
 \\
{order}              &\multicolumn{1}{c}{-}   &4.088174131992373	&3.927860511095907	&3.991692219748013	&3.998342489474011	&3.999488774370981
\\

\bottomrule
\end{tabular}

\end{table}

\section{Conclusion}\label{section5}

\noindent
This paper investigates time-splitting methods for the semiclassical Dirac equation. Through rigorous theoretical analysis, we prove that both the first-order and second-order time-splitting schemes preserve the mass conservation property of the Dirac equation. Furthermore, rigorous error estimates and convergence analysis are established for the proposed schemes. Theoretical results show that the proposed time-splitting methods achieve spectral accuracy in space, 
with temporal error bounds of 
$O\left(\tau/\varepsilon^2\right), O\left(\tau^2/\varepsilon^3\right), O\left(\tau^4/\varepsilon^5\right)$
for $S_1$, $S_2$, and $S_{\mathrm{4c}}$, respectively.
Numerical experiments are presented to validate the theoretical results
and demonstrate the effectiveness of the proposed methods.

\section*{Data Availability Statement}
The implementation code and numerical datasets supporting this study have been deposited in a Zenodo repository. Access to the repository has been provided in the submission system.

\begin{acknowledgements}
H. Wang and J. Yin were partially supported by Shanghai Rising-Star Program under Grant No. 24QA2700600 and National Natural Science Foundation of China under Grant No. 12571422. 
\end{acknowledgements}

\bibliographystyle{plain}
\bibliography{references_article_new}

@article{dirac1928quantum,
  author  = {Dirac, P. A. M.},
  title   = {{The quantum theory of the electron}},
  journal = {{Proc. R. Soc. Lond. A}},
  volume  = {117},
  number  = {778},
  pages   = {610--624},
  year    = {1928},
  publisher = {{The Royal Society London}}
}

@article{novoselov2004electric2,
  author  = {Novoselov, K. S. and Geim, A. K. and Morozov, S. V. and Jiang, D.-e. and Zhang, Y. and Dubonos, S. V. and Grigorieva, I. V. and Firsov, A. A.},
  title   = {{Electric field effect in atomically thin carbon films}},
  journal = {{Science}},
  volume  = {306},
  number  = {5696},
  pages   = {666--669},
  year    = {2004},
  publisher = {{American Association for the Advancement of Science}}
}

@article{castro2009electronic3,
  author  = {Castro Neto, A. H. and Guinea, F. and Peres, N. M. R. and Novoselov, K. S. and Geim, A. K.},
  title   = {{The electronic properties of graphene}},
  journal = {{Rev. Mod. Phys.}},
  volume  = {81},
  number  = {1},
  pages   = {109--162},
  year    = {2009},
  publisher = {{APS}}
}

@article{Hasan2010Topological4,
  author  = {Hasan, M. Z. and Kane, C. L.},
  title   = {{Colloquium: topological insulators}},
  journal = {{Rev. Mod. Phys.}},
  volume  = {82},
  number  = {4},
  pages   = {3045--3067},
  year    = {2010},
  publisher = {{APS}}
}

@article{Young2012Dirac5,
  author  = {Young, S. M. and Zaheer, S. and Teo, J. C. Y. and Kane, C. L. and Mele, E. J. and Rappe, A. M.},
  title   = {{Dirac semimetal in three dimensions}},
  journal = {{Phys. Rev. Lett.}},
  volume  = {108},
  number  = {14},
  pages   = {140405},
  year    = {2012},
  publisher = {{APS}}
}

@article{DiPiazza2012Strong6,
  author  = {Di Piazza, A. and M{\"u}ller, C. and Hatsagortsyan, K. Z. and Keitel, C. H.},
  title   = {{Extremely high-intensity laser interactions with fundamental quantum systems}},
  journal = {{Rev. Mod. Phys.}},
  volume  = {84},
  number  = {3},
  pages   = {1177--1228},
  year    = {2012},
  publisher = {{APS}}
}

@article{boada2011dirac7,
  author  = {Boada, O. and Celi, A. and Latorre, J. I. and Lewenstein, M.},
  title   = {{Dirac equation for cold atoms in artificial curved spacetimes}},
  journal = {{New J. Phys.}},
  volume  = {13},
  number  = {3},
  pages   = {035002},
  year    = {2011},
  publisher = {{IOP Publishing}}
}

@article{bao2017numerical,
  author  = {Bao, W. and Cai, Y. and Jia, X. and Tang, Q.},
  title   = {{Numerical methods and comparison for the Dirac equation in the nonrelativistic limit regime}},
  journal = {{J. Sci. Comput.}},
  volume  = {71},
  number  = {3},
  pages   = {1094--1134},
  year    = {2017},
  publisher = {{Springer}}
}

@article{bao2019fourth,
  author  = {Bao, W. and Yin, J.},
  title   = {{A fourth-order compact time-splitting Fourier pseudospectral method for the Dirac equation}},
  journal = {{Res. Math. Sci.}},
  volume  = {6},
  number  = {1},
  pages   = {11},
  year    = {2019},
  publisher = {{Springer}}
}

@article{fdtd_dirac,
  author  = {Ma, Y. and Yin, J.},
  title   = {{Error bounds of the finite difference time domain methods for the Dirac equation in the semiclassical regime}},
  journal = {{J. Sci. Comput.}},
  volume  = {81},
  number  = {3},
  pages   = {1801--1822},
  year    = {2019},
  publisher = {{Springer}}
}

@article{sp_SchrodingerEquation,
  author  = {Bao, W. and Jin, S. and Markowich, P. A.},
  title   = {{On time-splitting spectral approximations for the Schr{\"o}dinger equation in the semiclassical regime}},
  journal = {{J. Comput. Phys.}},
  volume  = {175},
  number  = {2},
  pages   = {487--524},
  year    = {2002},
  publisher = {{Elsevier}}
}

@article{yin_sp_time-dependent,
  author  = {Yin, J.},
  title   = {{A fourth-order compact time-splitting method for the Dirac equation with time-dependent potentials}},
  journal = {{J. Comput. Phys.}},
  volume  = {430},
  pages   = {110109},
  year    = {2021},
  publisher = {{Elsevier}}
}

@article{classical-solution-fdtd1,
  author  = {Antoine, X. and Lorin, E. and Sater, J. and Fillion-Gourdeau, F. and Bandrauk, A. D.},
  title   = {{Absorbing boundary conditions for relativistic quantum mechanics equations}},
  journal = {{J. Comput. Phys.}},
  volume  = {277},
  pages   = {268--304},
  year    = {2014},
  publisher = {{Elsevier}}
}

@article{classical-solution-fdtd2,
  author  = {Braun, J. W. and Su, Q. and Grobe, R.},
  title   = {{Numerical approach to solve the time-dependent Dirac equation}},
  journal = {{Phys. Rev. A}},
  volume  = {59},
  number  = {1},
  pages   = {604},
  year    = {1999},
  publisher = {{APS}}
}

@article{sp-Stokes-equation,
  author  = {Carelli, E. and Hausenblas, E. and Prohl, A.},
  title   = {{Time-splitting methods to solve the stochastic incompressible Stokes equation}},
  journal = {{SIAM J. Numer. Anal.}},
  volume  = {50},
  number  = {6},
  pages   = {2917--2939},
  year    = {2012},
  publisher = {{SIAM}}
}

@article{sp-semilag-SchrodingerEquation,
  author  = {Jin, S. and Zhou, Z.},
  title   = {{A semi-Lagrangian time splitting method for the Schr{\"o}dinger equation with vector potentials}},
  journal = {{Commun. Inf. Syst.}},
  volume  = {13},
  number  = {3},
  pages   = {247--289},
  year    = {2013}
}

@article{sp-semiregime-SchrodingerEquation,
  author  = {Bao, W. and Jin, S. and Markowich, P. A.},
  title   = {{Numerical study of time-splitting spectral discretizations of nonlinear Schr{\"o}dinger equations in the semiclassical regimes}},
  journal = {{SIAM J. Sci. Comput.}},
  volume  = {25},
  number  = {1},
  pages   = {27--64},
  year    = {2003},
  publisher = {{SIAM}}
}

@article{sp-dirac-uniform-error-bounds,
  author  = {Bao, W. and Feng, Y. and Yin, J.},
  title   = {{Improved uniform error bounds on time-splitting methods for the long-time dynamics of the Dirac equation with small potentials}},
  journal = {{Multiscale Model. Simul.}},
  volume  = {20},
  number  = {3},
  pages   = {1040--1062},
  year    = {2022},
  publisher = {{SIAM}}
}

@article{spectral-error,
  author  = {Pasciak, J. E.},
  title   = {{Spectral and pseudospectral methods for advection equations}},
  journal = {{Math. Comp.}},
  volume  = {35},
  number  = {152},
  pages   = {1081--1092},
  year    = {1980}
}

@incollection{hall2015baker,
  author    = {Hall, B.},
  title     = {{The Baker--Campbell--Hausdorff formula and its consequences}},
  booktitle = {{Lie Groups, Lie Algebras, and Representations: An Elementary Introduction}},
  pages     = {109--137},
  year      = {2015},
  publisher = {{Springer}}
}

@article{chin2001fourth,
  author  = {Chin, S. A. and Chen, C.-R.},
  title   = {{Fourth order gradient symplectic integrator methods for solving the time-dependent Schr{\"o}dinger equation}},
  journal = {{J. Chem. Phys.}},
  volume  = {114},
  number  = {17},
  pages   = {7338--7341},
  year    = {2001},
  publisher = {{American Institute of Physics}}
}

@article{strang-splitting-1968,
  author  = {Strang, G.},
  title   = {{On the construction and comparison of difference schemes}},
  journal = {{SIAM J. Numer. Anal.}},
  volume  = {5},
  number  = {3},
  pages   = {506--517},
  year    = {1968},
  publisher = {{SIAM}}
}

@article{bao2016uniformly,
  author = {Bao, W. and Cai, Y. and Jia, X. and Tang, Q.},
  title = {{A Uniformly Accurate Multiscale Time Integrator Pseudospectral Method for the Dirac Equation in the Nonrelativistic Limit Regime}},
  journal = {{SIAM J. Numer. Anal.}},
  volume = {54},
  number = {3},
  pages = {1785--1812},
  year = {2016}
}

\end{document}